\documentclass[a4paper, 12pt, reqno]{amsart}
\usepackage{amssymb}
\usepackage{etoolbox}
\usepackage[colorlinks=true,urlcolor=blue, citecolor=red,linkcolor=blue,linktocpage,pdfpagelabels, bookmarksnumbered,bookmarksopen]{hyperref}
\makeatletter
\def\@seccntformat#1{%
  \protect\textup{\protect\@secnumfont
    \ifnum\pdfstrcmp{section}{#1}=0 \bfseries\fi
    \csname the#1\endcsname
    \protect\@secnumpunct
  }%
}
\makeatletter
\def\@seccntformat#1{%
  \protect\textup{\protect\@secnumfont
    \ifnum\pdfstrcmp{section}{#1}=0 \bfseries\fi
    \ifnum\pdfstrcmp{subsection}{#1}=0 \bfseries\fi
    \csname the#1\endcsname
    \protect\@secnumpunct
  }%
}
\makeatother
\catcode`@=11
\def\section{\@startsection{section}{1}%
  \z@{1.5\linespacing\@plus\linespacing}{.5\linespacing}%
  {\normalfont\bfseries\large\centering}}
\catcode`@=12

\makeatletter
\def\@settitle{\begin{center}%
  \baselineskip14\p@\relax
    \normalfont\LARGE
\uppercasenonmath\@title  
  \@title
  \end{center}%
}
\makeatother
\makeatletter
\newtheorem{thm}{Theorem}[section]
\newtheorem{prop}[thm]{Proposition}
\newtheorem{lem}[thm]{Lemma}

\theoremstyle{remark}
\newtheorem{remark}[thm]{Remark}
\theoremstyle{definition}
\newtheorem{definition}[thm]{Definition}

\numberwithin{equation}{section}
\newcommand{\R}{\mathbb R}

\newcommand{\h}{{\mathcal H}}
\newcommand{\Y}{{\mathcal Y}}

\newcommand{\T}{\mathbb T}

\newcommand{\C}{\mathbb C}

\makeatletter
\def\@settitle{\begin{center}%
  \baselineskip16\p@\relax
    \bfseries
    \normalfont\normalsize
\uppercasenonmath\@title
  \@title
  \end{center}%
}
\makeatother
\begin{document}
\title[Multicomponent Schr\"{o}dinger-\MakeLowercase{g}K\MakeLowercase{d}V Systems]{\bf Well-posedness for multicomponent Schr\"{o}dinger--\MakeLowercase{g}K\MakeLowercase{d}V systems and stability of solitary waves with prescribed mass}

\author[S.~Bhattarai]{\small{Santosh Bhattarai}}
\author[A.~Corcho]{\small{Ad\'{a}n J.~Corcho}}
\author[M.~Panthee]{\small{Mahendra Panthee}}
\address[S.\ Bhattarai]{Trocaire College, 360 Choate Ave, Buffalo, NY 14220, USA}
\email{sntbhattarai@gmail.com}
\address[A.~J.~Corcho]{Instituto de Matem\'atica - Universidade Federal do Rio de Janeiro/UFRJ, Rio de Janeiro-RJ, Brazil}
\email{adan@im.ufrj.br}
\address[M.~Panthee]{IMECC - UNICAMP, 13083-859, Campinas, SP, Brazil}
\email{mpanthee@ime.unicamp.br}

\thanks{\textit{Mathematics Subject Classification}. 35Q53, 35Q55, 35B35, 35B65, 35A15.}
\thanks{\textit{Keywords.} Schr\"{o}dinger-KdV equations, local and global well-posedness, smoothing effects, Bourgain space, normalized solutions, solitary waves, stability, variational methods}

\begin{abstract}
In this paper
we prove the well-posedness issues of the associated initial value problem, the existence of nontrivial solutions with
prescribed $L^2$-norm, and the stability of associated solitary waves for
two classes of coupled nonlinear dispersive equations. The first problem here describes
the nonlinear interaction between two Schr\"{o}dinger type short waves and a generalized  Korteweg-de Vries type long wave and the second
problem describes the nonlinear interaction of two generalized Korteweg-de Vries type long waves with a common Schr\"{o}dinger type short wave.
The results here
extend many of the previously obtained results for two-component coupled Schr\"{o}dinger-Korteweg-de Vries systems.
\end{abstract}

 \maketitle

\tableofcontents

\section{Introduction}

\noindent In this paper, consideration is given to multicomponent nonlinear systems
describing the interaction between long and short dispersive waves. First we are concerned
with a 3-wave system
describing the interaction of two nonlinear Schr\"{o}dinger (NLS)-type short waves
with a generalized Korteweg-de Vries (gKdV)-type long wave and the second system we study contains
two gKdV-type long wave fields and a common NLS-type short wave.
The first problem considered here has the form
\begin{equation}\label{3LSI}
\ \ \ \ \ \ \ \ \left\{
\begin{aligned}
& i\partial_t u_1 + \partial_{x}^{2} u_1+ \gamma_1 |u_1|^{q_1}u_1 = -\alpha_1 u_1 v, \\
& i\partial_t u_2 + \partial_{x}^{2} u_2+ \gamma_2 |u_2|^{q_2}u_2 = -\alpha_2 u_2 v, \\
& \partial_t v + \partial_{x}^{3} v + \beta v^p \partial_x v = -\frac{1}{2}\partial_{x} (\alpha_1|u_1|^2+\alpha_2|u_2|^2),
\end{aligned}
\right.
\end{equation}
and the second system generally takes the form
\begin{equation}\label{3LSINKK}
\left\{
\begin{aligned}
& i\partial_t u + \partial_{x}^{2} u+ \gamma |u|^{q}u = -\alpha_1 u v_1 - \alpha_2 u v_2, \\
& \partial_t v_1 + \partial_{x}^{3} v_1 + \beta_1 v_1^{p_1} \partial_x v_1 = -\frac{1}{2}\alpha_1 \partial_{x} (|u|^2),\\
& \partial_t v_2 + \partial_{x}^{3} v_2 + \beta_2 v_2^{p_2} \partial_x v_2 = -\frac{1}{2}\alpha_2 \partial_{x} (|u|^2),
\end{aligned}
\right.
\end{equation}
where $u, u_1,$ and $u_2$ are $\mathbb{C}$-valued
functions of $(x,t)\in \mathbb{R}^2;$ $v, v_1,$ and $v_2$ are $\mathbb{R}$-valued functions of $(x,t)\in \mathbb{R}^2;$
and the constants $\alpha_j, \gamma_j, \gamma,$ $\beta_j$ and $\beta$ are reals which depend on
the context in which the system of equations have been derived. Here $v, v_1, v_2$ characterize
long-wave fields and $u, u_1, u_2$ represent short wave envelopes.
This type of phenomenon has been predicted in a variety of contexts in fluid mechanics, plasma physics,
nonlinear optics, acoustics, to mention but a few (for an excellent list of references, the reader may consult \cite{[AA],[San3LS]}).
Throughout this paper we refer to the systems \eqref{3LSI}  and \eqref{3LSINKK} simply as $(2+1)$-component NLS-gKdV and
$(1+2)$-component NLS-gKdV systems, respectively.

\smallskip

\noindent System \eqref{3LSI} admits three conserved quantities, i.e., time independent quantities, which will play an important role in this paper.
The first conserved quantity for \eqref{3LSI} is the energy functional $E$ defined by
\begin{equation}\label{Edef}
E(\Delta)=\int_{-\infty}^{\infty} \left( \sum_{j=1}^{2}\left(|\partial_xu_{j}|^2-\tau_j|u_j|^{q_j+2}-\alpha_j|u_j|^2v \right)+ |\partial_xv|^2-\tau v^{p+2} \right)\ dx
\end{equation}
where $\Delta=(u_1,u_2,v),$ and $\tau_j, 1\leq j\leq 2,$ and $\tau$ are given by
\begin{equation*}
\tau_j=\frac{2\gamma_j}{q_j+2}\ \ \textrm{and}\ \
\tau=\frac{2\beta}{(p+1)(p+2)}.
\end{equation*}
Other two conserved quantities for the flow defined by \eqref{3LSI} are
\begin{equation}\label{con-2}
H(\Delta)=\int_{-\infty}^{\infty} v^{2} \ dx + 2~\textrm{Im} \int_{-\infty}^{\infty}\sum_{j=1}^{2} u_j~\overline{\partial_xu_j} \ dx,
\end{equation}
where the bar denotes complex conjugation and $\textrm{Im}$ denotes the imaginary part of the complex function, and the component masses
\begin{equation}\label{con-1}
Q(u_j)=\int_{-\infty}^{\infty} |u_j|^2 \ dx,\ j=1,2.
\end{equation}

\smallskip

\noindent The first conserved quantity for \eqref{3LSINKK} is the energy functional $K$ defined by
\begin{equation}\label{EdefNKK}
K(U)=\int_{-\infty}^{\infty} \left( \sum_{j=1}^{2}\left(|\partial_xv_{j}|^2-b_j v_j^{p_j+2}-\alpha_j|u|^2v_j \right)+ |\partial_xu|^2-a |u|^{q+2} \right)\ dx
\end{equation}
where $U=(u,v_1,v_2),$ and $b_j, 1\leq j\leq 2,$ and $a$ are given by
\begin{equation*}
b_j=\frac{2\beta_j}{(p_j+1)(p_j+2)}\ \ \textrm{and}\ \
a=\frac{2\gamma}{q+2}.
\end{equation*}
Other two conservation laws of \eqref{3LSINKK} associated with symmetries are
\begin{equation}\label{con-1NKK}
G(U)=\int_{-\infty}^{\infty} \left(v_1^{2}+v_2^{2}\right) \ dx + 2~\textrm{Im} \int_{-\infty}^{\infty} u~\overline{\partial_xu} \ dx,
\end{equation}
which arises from the invariance of \eqref{3LSINKK} under space translations $x\to x+\theta,$
and the component mass
\begin{equation}\label{con-2NKK}
Q(u)=\int_{-\infty}^{\infty} |u|^2 \ dx,
\end{equation}
which arises from the invariance of \eqref{3LSINKK} under phase shifts $u\to e^{i\theta}u.$

\smallskip
\noindent The first purpose of this paper is to consider the question of well-posedness of the initial value
problem (IVP) associated to  the systems \eqref{3LSI} and \eqref{3LSINKK}. We adapt the standard notion
of the well-posedness in the sense of J. Hadamard, which includes existence, uniqueness, persistence property (i.e., the solution is uniquely determined and it has the
same regularity as the initial data), and continuous dependence of the solution upon the given data.

\smallskip

\noindent The IVP asociated to the $(1+1)$-component NLS-KdV system has been studied extensively in the literature.
In the case when $u_2\equiv0,\ 2p=q_1=2,$
$\beta=1,$ and $\gamma_1\in \mathbb{R},$ the local well-posedness was studied in \cite{[Tsu],BOP}.
Here the cases $\gamma_1=0$ and $\gamma_1\neq 0$  describe the resonant and non-resonant interactions, respectively.
In the resonant case,
Guo and Miao \cite{Guo} established the global well-posedness result of $(1+1)$-component NLS-KdV system
in the energy space $H^{1}(\mathbb{R})\times H^{1}(\mathbb{R}).$ In \cite{[Pec]}, Pecher improved these
results and obtained the local well-posedness for the data
in $H^{s}(\mathbb{R})\times H^{s}(\mathbb{R})$ with $s>0$ and the global-posedness
for $({u_1}_0,v_0)\in H^{\frac{3}{5}+}(\mathbb{R})\times H^{\frac{3}{5}+}(\mathbb{R})$
in the resonant case and for $({u_1}_0,v_0)\in H^{\frac{2}{3}+}(\mathbb{R})\times H^{\frac{2}{3}+}(\mathbb{R})$
in the non-resonant case. In \cite{CL}, Corcho and Linares improved the
local well-posedness result obtained in \cite{[Pec]} to a larger region of the Sobolev indices.
Recently, Wu \cite{WU} obtained the best local well-posedness result for the $(1+1)$-component NLS-KdV system in the resonant case.
Our aim here is to obtain analogous results to the full system of
equations \eqref{3LSI} and \eqref{3LSINKK}, considering general power nonlinearities, in the Sobolev spaces of the form $H^s\times H^{s}\times H^{k}$ and $H^s\times H^{k}\times H^{k},$ respectively.

\smallskip

\noindent The well-posedness issues here are addressed considering two different cases, viz., general power type  and integer power type nonlinearities. For the general power type  nonlinearities, we use smoothing effects of the associated linear groups combined with the maximal function type estimates to prove the local well-posedness in the energy space $H^1\times H^1\times H^1$. Also, with certain restriction on the indices of nonlinearity, we obtain global solution in this space. To be precise,  the lack of $L^2$-conserved quantity for the gKdV part requires this restriction on the power of nonlinearities (see Theorems \ref{Thm1} and \ref{Thm2} and their proofs below).
In the particular case when the indices of nonlinearities $q_1=q_2=2p=2$ for the system \eqref{3LSI} and $2p_1=2p_2=q=2$ for the system \eqref{3LSINKK}, we use the estimates obtained  in \cite{WU} in the framework of Bourgain spaces to get local well-posedness results for the less regular data (see Theorems~\ref{Thm1.1} and \ref{Thm1.2}).

\smallskip

\noindent Next, attention will be focused to prove the
existence of nontrivial (i.e., all components non-zero)
solutions $(\sigma_1,\sigma_2,c,\phi_1,\phi_2,w)\in \mathbb{R}_{+}^2\times \mathbb{R}\times \h$
of the system of equations
\begin{equation}\label{ODE}
\left\{
\begin{aligned}
& -\phi_1^{\prime \prime}+\sigma_1 \phi_1=\gamma_1|\phi_1|^{q_1}\phi_1+\alpha_1\phi_1 w, \\
& -\phi_2^{\prime \prime}+\sigma_2 \phi_2=\gamma_2|\phi_2|^{q_2}\phi_2+\alpha_2\phi_2 w, \\
& -w^{\prime \prime}+c w=\frac{\beta}{p+1}w^{p+1}+\sum_{j=1}^{2}\frac{\alpha_j}{2}|\phi_j|^2.
\end{aligned}
\right.
\end{equation}
System of ordinary differential equations \eqref{ODE} can be considered as
the defining equation for travelling solitary waves of \eqref{3LSI}. Solitary waves of interest here have the form
\begin{equation}\label{SO}
\left\{
\begin{aligned}
& u_1(x,t)=e^{i\omega_1 t}e^{ic(x-ct)/2}\phi_1(x-ct) ,\\
& u_2(x,t)= e^{i\omega_2 t}e^{ic(x-ct)/2}\phi_2(x-ct),\\
& v(x,t) =w(x-ct),
\end{aligned}
\right.
\end{equation}
where $\phi_1,\phi_2:\mathbb{R}\to \mathbb{C},\ w:\mathbb{R}\to \mathbb{R}$ all vanish
at $\pm \infty,$
and the parameters $\omega_1, \omega_2, c$ are reals.
Substituting
solitary waves ansatz \eqref{SO} into \eqref{3LSI},
one easily finds that $(\phi_1,\phi_2,w)$ satisfies the
time-independent 3-component NLS-gKdV system \eqref{ODE} with $\xi=x-ct$ and $\sigma_j=\omega_j-c^2/4.$

\smallskip

\noindent Given any $(r,l,m)\in \mathbb{R}_{+}^3,$
we look for solutions $(\phi_1,\phi_2,w)$ of \eqref{ODE} satisfying the condition
\begin{equation}\label{norcon2}
\|\phi_1\|_{L^2}^2 =r,\ \|\phi_2\|_{L^2}^2 =l,\ \textrm{and}\ \|w\|_{L^2}^2 =m.
\end{equation}
These type of solutions are of particular interest in physics.
In the literature, these solutions are sometimes referred to as $L^2$-normalized solutions.
To infer the existence of such solutions, we study
the constrained variational problem of finding, for given $(r, l, m)\in \mathbb{R}_{+}^3,$
the extremum of the functional $E$ over the set $S_r\times S_l \times K_m,$ where for any $\lambda>0$ we define
\begin{equation*}
S_\lambda:=\left\{u\in H_{\mathbb{C}}^1:\|u\|_{L^{2}}^2=\lambda \right\}\ \textrm{and}\ K_\lambda:=\left\{u\in H_{\mathbb{R}}^1:\|u\|_{L^{2}}^2=\lambda \right\}.
\end{equation*}
The key ingredient in the proof the existence of minimizers is the concentration compactness lemma
introduced by P.L. Lions \cite{[L1]}. The parameters $\sigma_1, \sigma_2,$ and $c,$
in this situation, appear as Lagrange multipliers associated with the constraints.

\smallskip

\noindent Several work has been done in the last few years on the existence problem for solutions
of coupled nonlinear systems such as \eqref{ODE}.
All these works have been mainly focused on $(1+1)$-component coupled
systems such as NLS-NLS and NLS-KdV systems.
Moreover, most works treat the problem in which the parameters such as $\sigma_1, \sigma_2,$ $c$ are being fixed.
There are very few papers which deal with the existence problem of
prescribed $L^2$-norm solutions, for instance, see \cite{[AA], [AB11], [Bar],[Gar1], [Noris]} for the results on
prescribed $L^2$-norm solutions to two-component coupled systems.
Up to our knowledge, \cite{[SanDCDS],[Ikoma]} are the only available works
which obtain prescribed $L^2$-norm solutions for coupled nonlinear systems with three or more equations.
The techniques in \cite{[SanDCDS]} follow the ideas used in
\cite{[AB11]} to obtain existence and stability
results of $L^2$-normalized solitary waves for three component nonlinear Schr\"{o}dinger system.
In \cite{[Ikoma]}, a different technique was used to prove the stability of the set
of minimizers to a certain minimizing problem under multiconstraint conditions.
In the present paper the situation is substantially
different compared to that of \cite{[SanDCDS], [Ikoma]} due to the presence of the additional conserved quantity $H(f,g,h).$
Here we need to
tackle two different variational problems in order to establish the stability result.
Finally, we also mention the papers
\cite{[San3LS],[Colo]} where different techniques were
used to prove the existence of bound state solutions for multi-component NLS-KdV systems.

\smallskip

\noindent Our final goal is to study the stability properties of solitary wave solutions of  \eqref{3LSI}.
The mathematically exact stability theory for travelling solitary waves began
with a $1972$ paper of T.~B.~Benjamin (\cite{[Benj]}) for the KdV equation
\begin{equation}\label{kDV}
\partial_tu+u\partial_xu+\partial_x^3u=0.
\end{equation}
According to Benjamin, if $U(x,t)$ is a solution of \eqref{kDV} whose
initial profile $U(x,0)=U_0(x)$ is sufficiently close (in an appropriate function space)
to a KdV solitary wave $u(x,t)=\varphi_C(x-Ct),$ where $\varphi_C(x)$ defined as
\begin{equation*}
\varphi_C(x)=\frac{3C}{\textrm{cosh}^2\left(\frac{1}{2}\sqrt{C}~x\right)};
\end{equation*}
then the quantity
\begin{equation}\label{benjq}
\inf_{x\in \mathbb{R}}~\sup_{x\in \mathbb{R}}|U(x,t)-\varphi_C(x+y)|
\end{equation}
will remain small for all times $t\geq 0.$
Similar stability theorems have since been proved for solitary-wave solutions of
many other nonlinear wave equations. Notice that the quantity \eqref{benjq} measures the difference in sup norm between the
profile $u(x,t)$ for fixed $t$ and the orbit consisting of all translates of $\varphi_C.$
Since, for system \eqref{3LSI}, we do not know if for given phases $\omega_1,\omega_2$ and wave speed $c,$
solitary-wave solutions are unique up to translation,
we use the notion of stability in a broad sense: namely, the
stability of a set consisting of possibly different solitary-wave profiles functions
rather than the stability of the set of translates of a
single solitary-wave profile.
The precise details of our stability results are contained in Section~\ref{Mainresults} (see Theorem~\ref{P2thm} and \ref{stabilitytheorem}).

\medskip

The structure of the paper is as follows.
In Section~\ref{Mainresults}, we start with some notations that will be used throughout the paper and provide the statement of main results.
Section~\ref{LWP-Theory} addresses the issues of well-posedness theory.
In Section~\ref{VarProbP1}, we prove the existence result for $L^2$ normalized solitary-wave solutions for $(2+1)$-component NLS-gKdV system. Finally,
Section~\ref{fullvar} studies an alternative variational characterization of solitary waves, along with their stability properties.

\section{Statement of Main Results}\label{Mainresults}
In this section, we introduce some notations and function spaces that will be used throughout
the paper and state our main results.
\subsection{Notations and assumptions}
We denote by $\mathbb{R}_{+}$ the set $\{x\in \mathbb{R}:x>0\}$ and
by $S^1$ the set $\{z\in \mathbb{C}: z=e^{i\theta}, \theta \in \mathbb{R}\}.$
For $1\leq p\leq \infty,$ we denote by $L^p=L^{p}(\mathbb{R})$ the Banach space of Lebesgue measurable functions
on $\mathbb{R}$ with the usual norm $\|\cdot\|_{L^p}.$
For $s\in \R$, the $L^2$-based Sobolev space of order $s$ of   complex-valued
functions $f$ will be denoted by $H^s_{\C} = H^s_{\C}(\R)$
and the usual norm on this space is denoted by $\|\cdot\|_{H^s}$. More generally, if $B$ is
any Banach space the norm on $B$ will be denoted by $\|\cdot\|_B.$
We denote by $H_{\mathbb{R}}^s=H_{\mathbb{R}}^s(\mathbb{R})$ the space of all real-valued
functions $f$ in $H_{\mathbb{C}}^s$ and $H_{+}^s(\mathbb{R})$ denotes the
space of all functions $f$ in $H_{\mathbb{R}}^s$ such that $f(x)>0$ on $\mathbb{R}.$
If $B_1$ and $B_2$ are Banach spaces, then their Cartesian product $B_1\times B_2$ is a Banach space with a product
norm defined by $\|(f,g)\|_{B_1\times B_2}:=\|f\|_{B_1}+\|g\|_{B_2}.$ In particular,
we define
\[
\h=H_{\mathbb{C}}^1\times H_{\mathbb{C}}^1\times H_{\mathbb{R}}^1\ \ \textrm{and}\ \ \Y=H_{\mathbb{C}}^1\times H_{\mathbb{R}}^1\times H_{\mathbb{R}}^1.
\]
If $B$ is a Banach space and $G$ is a subset of $B,$ we say that a sequence $\{x_n\}$ in $B$ converges to $G$ if
\begin{equation*}
\lim_{n\to \infty}\ \inf_{g\in G}\|x_n-g\|_B=0.
\end{equation*}
Also, for each $T>0,$ we denote by $\mathcal{C}([0,T]; \ B)$ the Banach space
of continuous maps $f$ from $[0,T]$ to $B,$ with norms given by
\begin{equation*}
\|f\|_{\mathcal{C}([0,T]; \ B)} = \sup_{t\in [0,T]}\|f(t)\|_B.
\end{equation*}
\noindent For any $a\in \mathbb{R},$ we denote by $T_a$ the translation operator defined by $(T_af)(\cdot)=f(\cdot+a)$.
Also, we use notation $A_1\lesssim A_2$ if there exist constants $C_1$ and $C_2$ such that $A_1\leq C_1A_2$ and $A_2\leq C_2A_1$.
The symbol $C$ will be used throughout to denote various
constants whose exact values are not important and which may differ from one line to the
next.

\smallskip

\noindent Following standard notations in the literature,  $D^s$ and  $J^s$, respectively, denote the multiplication operators (via the Fourier transform) with symbols $|\xi|^s$ and $(1+\xi^2)^{s/2}$. Thus, $J^{-s}$ is the usual Bessel potential and the classical Sobolev space in the line is defined by  $H^s=J^{-s}(L^2)$ with $\|\varphi\|_{H^s}=\|J^s\varphi\|_{L^2}$.  Also, throughout the work for any $1\le p \le \infty$ we denote by $p'$ the exponent such that $\frac1p+\frac{1}{p'}=1$ and we will used the space-time Lebesgues spaces $L^{\rho}_TL^{\nu}_x$ and
$L^{\nu}_xL^{\rho}_T$ equipped with the norms
\begin{align}
& \|f(x,t)\|_{L^{\rho}_TL^{\nu}_x}=\big\|\,\| f(\cdot,t) \|_{L^{\nu}(\mathbb{R})}\,\big\|_{L^{\rho}([0, T])},\nonumber\\
& \|f(x,t)\|_{L^{\nu}_xL^{\rho}_T}=\big\|\,\| f(x,\cdot) \|_{L^{\rho}([0, T])}\,\big\|_{L^{\nu}(\mathbb{R})}.\nonumber
\end{align}

We use $S(t)$ and $V(t)$ given by
\begin{equation}\label{lin-prop}
S(t)=e^{it\partial_x^2}\quad\textrm{and}\quad V(t)=e^{-t\partial_x^3},
\end{equation}
 to denote  the linear propagators for the Schr\"{o}dinger and the KdV equations respectively.
Given $s,k\in \R$ and $0<b<1$, we define two function spaces $X^{s,b}$ and $Y^{k,b}$ as the completion of the Schwartz space $\mathcal{S}(\R^n)$ with respect to the norms
\begin{equation}\label{xsb-norm}
\begin{split}
\|f\|_{X^{s,b}}&:=\int\!\!\!\int\langle\xi\rangle^{2s}\langle\tau+\xi^2\rangle^{2b} |\widehat{f}(\xi,\tau)|^2d\tau d\xi\\
&=\|S(-t)f\|_{H_t^b(\R;H_x^s)}
\end{split}
\end{equation}
and
\begin{equation}\label{ykb-norm}
\begin{split}
\|f\|_{Y^{k,b}}&:=\int\!\!\!\int\langle\xi\rangle^{2k}\langle\tau-\xi^3\rangle^{2b} |\widehat{f}(\xi,\tau)|^2d\tau d\xi\\
&=\|V(-t)f\|_{H_t^b(\R;H_x^k)},
\end{split}
\end{equation}
where $\langle\cdot\rangle :=1+|\cdot|$.

Finally, we introduce the following even smooth cut-off function $\psi \in C_0^{\infty}(\R)$ given by
\begin{equation}\label{cut-off}
\psi(t)=\begin{cases} 1, \quad\textrm{for}\ |t|\leq 1,\\
0, \quad\textrm{for}\ |t|\geq 2,
\end{cases}
\end{equation}
and  define $\psi_T(t):=\psi(\frac t{T})$.

\smallskip

We now state our main results.
\subsection{Well-posedness results}

Here we state the main results about  well-posedness theory established in this work for the IVPs associated to the systems \eqref{3LSI} and \eqref{3LSINKK}.

\begin{thm}\label{Thm1}
Consider system \eqref{3LSI} with $1\le p=\frac{n_1}{n_2}\in \mathbb{Q}^+$,  $n_2$ odd,  and $q_j>0$ for $j=1,2$. For any given data $({u_1}_0, {u_2}_0, v_0) \in H^1\times H^1 \times H^1$ there is a positive time $T=T(\|{u_1}_0\|_{H^1}, \|{u_2}_0\|_{H^1}, \|v_0\|_{H^1})$ and a unique solution
$(u_1, u_2, v)$ to the IVP associated to to \eqref{3LSI} such that
\begin{align}
&(u_1, u_2, v)\in \mathcal{C}([0, T];\, H^1\times H^1\times H^1)\label{Thm1-cond-a}\\
&\|u_1\|_{L^2_xL^{\infty}_T} + \|u_2\|_{L^2_xL^{\infty}_T} +  \|v\|_{L^2_xL^{\infty}_T}\lesssim (1+T)^{\frac3{4}+}.\label{Thm1-cond-b}
\end{align}
Moreover, for any $T' < T$ there exists a neighborhood $\mathcal{V}$  of $(u_{1_0}, u_{2_0}, v_0)$ in $ H^1\times H^1 \times H^1$ such that the map
$$(\tilde{u_1}_0, \tilde{u_2}_0, \tilde{v}_0)\longmapsto (\tilde{u}_1, \tilde{u}_2, \tilde{v})$$
from  $\mathcal{V}$ into the class defined by \eqref{Thm1-cond-a}-\eqref{Thm1-cond-b} with $T'$ instead of $T$ is continuous. Also, for $1\leq p<\frac43$
and $q_j$ for  $j=1,2$, verifying
\begin{equation}
0< q_j <
\begin{cases}
\;4     & \text{if}\quad \tau_j > 0,\\
\infty  & \text{if}\quad \tau_j \le 0,
\end{cases}	
\end{equation}
the local solution can be extended to any time interval $[0, T]$ with $T$ arbitrary large.
\end{thm}

\begin{thm}\label{Thm2}
	Consider system \eqref{3LSINKK} with $q>0$ and $1\le p_j=\frac{{n_j}_1}{{n_j}_2}\in \mathbb{Q}^+$, ${n_j}_2$ odd, for $j=1,2$. For any given data $(u, {v_1}_0, {v_2}_0) \in H^1\times H^1 \times H^1$ there is a positive time $T=T(\|u_0\|_{H^1}, \|{v_1}_0\|_{H^1}, \|{v_2}_0\|_{H^1})$ and a unique solution
	$(u, v_1, v_2)$ to the IVP associated to \eqref{3LSINKK} such that
	\begin{align}
	&(u, v_1, v_2)\in \mathcal{C}([0, T];\, H^1\times H^1\times H^1)\label{Thm2-cond-a}\\
	&\|u\|_{L^2_xL^{\infty}_T} + \|v_1\|_{L^2_xL^{\infty}_T} +  \|v_2\|_{L^2_xL^{\infty}_T}\lesssim (1+T)^{\frac3{4}+}.\label{Thm2-cond-b}
	\end{align}
	Moreover, for any $T' < T$ there exists a neighborhood $\mathcal{V}$  of $(u_0, {v_1}_0, {v_2}_0)$ in $ H^1\times H^1 \times H^1$ such that the map
	$$(\tilde{u}_0, \tilde{v_1}_0, \tilde{v_2}_0)\longmapsto (\tilde{u}, \tilde{v}_1, \tilde{v}_2)$$
	from  $\mathcal{V}$ into the class defined by \eqref{Thm2-cond-a}-\eqref{Thm2-cond-b} with $T'$ instead of $T$ is continuous.  Also, for $1\leq p_j<\frac43$, $j=1,2$,
	and $q$ verifying
	\begin{equation}
	0< q <
	\begin{cases}
	\;4     & \text{if}\quad a> 0,\\
	\infty  &\text{if}\quad a\le 0,
	\end{cases}	
	\end{equation}
	 the local solutions can be extended to any time interval $[0, T]$ with $T$ arbitrary large.
\end{thm}

As discussed in the introduction, for the special cases  $q_1=q_2=2p=2$ for system \eqref{3LSI} and $2p_1=2p_2=q=2$ for system \eqref{3LSINKK} we prove the following more general local well-posedness results using contraction mapping principle in the framework of Bourgain's  spaces. More precisely, in these cases we have the following local well-posedness theorems.

\begin{thm}\label{Thm1.1}
Consider system \eqref{3LSI} with $q_1=q_2=2p=2$. Let  $({u_1}_0, {u_2}_0, v_0)$ belonging to the space $H^{s_1}\times H^{s_2}\times H^{k}$ with  $k > -3/4$ provided:
\begin{itemize}
\item [(a)] $\max\{k-1,\, \kappa/4\}< s_j < \kappa +2$\; if\; $\gamma_j=0$,\; for\;  $j=1,2$,\medskip
\item [(b)] $\max\{k-1,\, k/4\}< s_j < \kappa +2$\; and\; $s_j\ge 0$\; if\; $\gamma_j\neq 0$,\; for\;  $j=1,2$.
\end{itemize}
Then, there exist a time $T(\|{u_1}_0\|_{H^{s_1}}, \|{u_2}_0\|_{H^{s_2}} ,\|v_0\|_{H^{k}})>0$ and a unique solution for the integral equations associated to the  IVP for \eqref{3LSI} in an appropriate Bourgain's space contained in $\mathcal{C}([0, T];\;  H^{s_1}\times H^{s_2}\times H^{k})$.

Moreover, the mapping
$({u_1}_0, {u_2}_0, v_0) \longmapsto (u_1(\cdot, t), u_2(\cdot, t), v(\cdot, t))$ is locally Lipschitz.
\end{thm}

\begin{thm}\label{Thm1.2}
Consider system \eqref{3LSINKK} with $2p_1=2p_2=q=2$. Let  $({u}_0, {v_1}_0, {v_2}_0)$ belonging to the space $H^{s}\times H^{k_1}\times H^{k_2}$ with $k_j > -3/4$,\, $j=1,2$, provided:
\begin{itemize}
\item [(a)] $s-2\le k_j < \min\{4s,\;s+1\}$\; if\; $\gamma=0$,\; for\;  $j=1,2$,\medskip
\item [(b)] $s-2\le k_j < \min\{4s,\;s+1\}$\; and\; $s\ge 0$\; if\; $\gamma\neq 0$,\; for\;  $j=1,2$.
\end{itemize}
Then, there exist a positive time $T(\|{u}_0\|_{H^s}, \|{v_1}_0\|_{H^{k_1}} ,\|{v_2}_0\|_{H^{k_2}})$ and a unique solution for the integral equations associated to the  IVP for \eqref{3LSINKK} in an appropriate Bourgain's space contained in $\mathcal{C}([0, T];\;  H^{s}\times H^{k_1}\times H^{k_2})$.\\
 Moreover, the mapping
$({u}_0, {v_1}_0, {v_2}_0) \longmapsto (u(\cdot, t), v_1(\cdot, t), v_2(\cdot, t))$ is locally Lipschitz.
\end{thm}

\subsection{Existence and stability results}
To state our existence theorem, let us
denote by $\mathcal{O}_{r,l,m}$ the set of all
normalized solutions $(\phi_1,\phi_2,w)$ of \eqref{ODE} satisfying the condition \eqref{norcon2}.
We assume that the following conditions hold:
\begin{equation}\label{assumptions}
\begin{cases}
\gamma_1, \gamma_2>0,\ \beta>0,\ \alpha_1, \alpha_2>0,\\
0< q_1, q_2<4,\ \textrm{and} \ p=\frac{n_1}{n_2}\in \mathbb{Q}^{+},\ n_2\ \textrm{odd}.
\end{cases}
\end{equation}
The following theorem guarantees that the set $\mathcal{O}_{r,l,m}$ is non-empty.

\begin{thm}\label{SOexistence}
Suppose the assumptions \eqref{assumptions} hold and that $0< p<4.$
Then for every $(r,l,m)\in \mathbb{R}_{+}^3,$ there exists
a solution
\begin{equation*}
(\sigma_{1r},\sigma_{2l},c_m,\phi_r,\phi_l,w_m)\in \mathbb{R}_{+}^2\times \mathbb{R} \times\h
\end{equation*}
to the system \eqref{ODE}
satisfying the condition
\begin{equation}\label{normSO}
\|\phi_r\|_{L^2}^2=r,\ \|\phi_l\|_{L^2}^2=l,\ \textrm{and}\ \|w_m\|_{L^2}^2=m
\end{equation}
with $\sigma_{1r},\sigma_{2l},$ and $c_m$ being the Lagrange multipliers.
Moreover, $w_m(x)> 0$ for all $x\in \mathbb{R}$ and there exists $(\zeta_j,R_j)\in S^1\times H_{+}^1(\mathbb{R})$
such that
\begin{equation*}
\phi_r(x)=\zeta_1R_1(x)\ \ \textrm{and}\ \ \phi_l(x)=\zeta_2R_2(x), \ \textrm{for all}\ x\in \mathbb{R}.
\end{equation*}
In particular, $(\sigma_{1r},\sigma_{2l},c_m, R_1,R_2,w_m)$ is a real-valued positive solution of \eqref{ODE}.
\end{thm}

\smallskip

\noindent Our approach to study the stability of solitary waves is purely variational.
For any $(r,l,m)\in \mathbb{R}_{+}^2\times \mathbb{R},$ we consider a new variational formulation of solitary waves, namely
the problem of minimizing the functional
$E(h_1,h_2,g)$ over the set
\begin{equation}\label{Pidef}
\Pi_{r,l,m}=\{\Delta \in \h:\Delta=(h_1,h_2,g), Q(h_1)=r,\ Q(h_2)=l,\ \textrm{and}\ H(\Delta)=m\}.
\end{equation}
The family of minimization problems
\begin{equation}\label{Vardef}
\Lambda(r,l,m)=\inf\{E(h_1,h_2,g): (h_1,h_2,g)\in \Pi_{r,l,m} \}
\end{equation}
is suitable for studying the stability properties of travelling solitary waves because both $E$ and the constraint
functionals $Q(h_j)$ and $H(h_1,h_2,g)$ are invariants of motion of
\eqref{3LSI}. For such a variational problem,
an easy consequence of application of the concentration compactness
argument \cite{[L1], CLi} is that the set of global minimizers forms a stable set for the associated initial-value problem, in
that a solution which is initially close to this set will remain close to it for later times.

\smallskip

\noindent The next result concerns the existence of solutions to the variational problem \eqref{Vardef} and their
relation with those in $\mathcal{O}_{r,l,m}.$

\begin{thm}\label{P2thm}
Suppose the assumptions \eqref{assumptions} hold and that $1\leq p<4/3.$ Then

\smallskip

\noindent (i) every minimizing
sequence $\{(h_{1n},h_{2n},g_n)\}_{n\geq 1}$ for $\Lambda(r,l,m)$ enjoys the following compactness property:
there exists a subsequence $\{(h_{1n_k},h_{2n_k},g_{n_k})\}_{k\geq 1},$ a family
$(y_k) \subset \mathbb{R},$ and a
function $(\Phi_1,\Phi_2,w)\in \h$ such that the translated subsequence
\begin{equation*}
\{(T_{y_k}h_{1n_k}, T_{y_k}h_{2n_k}, T_{y_k}g_{n_k})\}_{k\geq 1}
\end{equation*}
converges strongly to $(\Phi_1, \Phi_2,w)$ in $\h.$ The function $(\Phi_1, \Phi_2,w)$ achieves the minimum,
\begin{equation*}
(\Phi_1, \Phi_2,w)\in \Pi_{r,l,m}\ \ \textrm{and}\ \ E(\Phi_1, \Phi_2,w)=\Lambda(r,l,m).
\end{equation*}

\smallskip

\noindent (ii) if $(\Phi_1,\Phi_2,w)$ is
a solution of \eqref{Vardef} and $\|w\|_{L^2}^2=N,$ then there
exists $(\zeta_j,\phi_j)\in S^1\times H_{+}^1(\mathbb{R})$ such
that $(\phi_1,\phi_2,w)\in \mathcal{O}_{r,l,N}$ and
\begin{equation*}
\Phi_1(x)=\zeta_1 e^{-ib_{r,l,m}(N)x}\phi_1(x)\ \ \textrm{and}\ \ \Phi_2(x)=\zeta_2 e^{-ib_{r,l,m}(N)x}\phi_2(x),\ x\in \mathbb{R},
\end{equation*}
where $b_{r,l,m}(A)$ is defined by
\begin{equation}\label{bdef}
b_{r,l,m}(A)=\frac{m-A}{2r+2l}\ \textrm{for any}\ A\geq 0.
\end{equation}
Furthermore, if $\gamma_1=\gamma_2=0,$
then the function $w$ can be chosen to be strictly positive on $\mathbb{R}.$
\end{thm}

\smallskip

\noindent We use the following
notion of stability.
\begin{definition}\label{defstability}
For $(r,l,m)\in \mathbb{R}_{+}^2\times \mathbb{R},$ let $\mathcal{P}_{r,l,m}$ be the set of
solutions of \eqref{Vardef}.
We say that the set $\mathcal{P}_{r,l,m}$ of solitary-wave profiles is stable if
for all $\epsilon>0,$ there
exists $\delta>0$ such that for any initial datum
$\Delta_0=(h_{01},h_{02},g_0)$ satisfying
\begin{equation*}
\Delta_0 \in \h, \ \ \ \inf\left\{ \|(\Phi_1,\Phi_2,w)-\Delta_0\|_\h:(\Phi_1,\Phi_2,w)\in \mathcal{P}_{r,l,m} \right\}  < \delta,
\end{equation*}
then the solution $\Delta(t,x)=(u_1(t,x),u_2(t,x),v(t,x))$ of \eqref{3LSI} satisfies for all times $t\geq 0,$
\begin{equation*}
\inf\left\{ \|(\Phi_1,\Phi_2,w) -\Delta(t,\cdot)\|_\h: (\Phi_1,\Phi_2,w)\in \mathcal{P}_{r,l,m}\right\} < \epsilon.\
\end{equation*}
\end{definition}

Our stability result reads as follows.

\begin{thm}\label{stabilitytheorem}
Suppose the assumptions \eqref{assumptions} hold and that $1\leq p <4/3.$ Then the following statements hold:

\medskip

\noindent $(i)$ every minimizing sequence $\{(f_{1n},f_{2n},g_n)\}$
for \eqref{Vardef} converges to $\mathcal{P}_{r,l,m}$ in $\h.$

\medskip

\noindent $(ii)$ the set $\mathcal{P}_{r,l,m}$ of minimizers is stable in the sense as in Definition~\ref{defstability}.

\smallskip

\noindent $(iii)$ the set $\mathcal{P}_{r,l,m}$ of minimizers forms a true-three parameter family, that is, if the sets $\mathcal{P}_{r_i,l_i,m_i}$ contain
$(\Phi_1^{(i)},\Phi_2^{(i)},w^{(i)})$ for $(r_1,l_1,m_1)\neq (r_2,l_2,m_2),$ then
\begin{equation*}
(\Phi_1^{(1)},\Phi_2^{(1)},w^{(1)})\neq (\Phi_1^{(2)},\Phi_2^{(2)},w^{(2)}).
\end{equation*}
\end{thm}

\smallskip

\begin{remark}
Our stability result generalizes and extends analogous results
for (1+1)-component NLS-KdV solitary waves
previously obtained in \cite{[C]}, which
considered the particular case $u_2\equiv 0, \gamma_1=0, p=1,\ \textrm{and}\ \alpha_1=1/6;$
\cite{[AA]}, which studied the case when $u_2\equiv 0, \gamma_1=0, p=1,\ \textrm{and}\ \alpha_1$ in some
neighborhood of $1/6;$ \cite{[AB11]}, which considered the case when $u_2\equiv 0, \gamma_1>0, 1\leq q_1<4, p=1,$ and $\alpha_1>0;$
and of \cite{[AP1]}, which proved a stability result for certain sets of solitary waves in
the special case when $u_2\equiv 0, \gamma_1=0, p=1, \alpha_1>0,$ and the wavespeed $\sigma_1$ is near $c^2/4.$
\end{remark}

\begin{remark}
Implicit in the notion of stability above is the assumption that $(2+1)$-component NLS-gKdV is
 globally well-posed in the energy space $\h.$
Our global well-posedness theory requires the restriction $1\leq p<4/3$.
As far as we know, it remains an open question whether \eqref{3LSI} is well-posed
in $\h$ for $0<p<1.$ If one assumes that the $(2+1)$-component NLS-gKdV is globally well-posed in $\h$ for the range $0< p <4/3,$ then the conclusions of
Theorems~\ref{P2thm} and \ref{stabilitytheorem} continue to hold for these values of $p.$
\end{remark}

\begin{remark}
Although we do not pursue these topics here, similar problems related to (1+2)-component NLS-gKdV system are that of the existence and
stability results concerning nontrivial solitary wave solutions of the form
\begin{equation*}
\left(e^{i\lambda t}e^{i\sigma(x-\sigma t)/2}\phi(x-\sigma t),w_1(x-\sigma t),w_2(x-\sigma t) \right),
\end{equation*}
where the functions $\phi :\mathbb{R}\to \mathbb{C},\ w_1,w_2:\mathbb{R}\to \mathbb{R}$ vanish
at $\pm \infty,$
and the parameters $\lambda$ and $\sigma$ are real.
As in the (2+1)-component case, one can study the existence and stability questions of (1+2)-component NLS-gKdV solitary waves via
their variational characterizations.
To find a true two-parameter family of travelling solitary wave solutions (parameterized by $\sigma$ and $\lambda$),
one considers
the two-parameter variational problem
\begin{equation}\tag{P2}
\inf\left\{K(U): U=(f,g_1,g_2)\in \Y:\ \|f\|_{L^2}^2=l>0\ \textrm{and}\ \sum_{j=1}^2\|g_j\|_{L^2}^2=r>0 \right\}.
\end{equation}
As usual in the method of concentration compactness, putting the method
into practice requires verifying the strict subadditivity condition for the function defined by (P2) with
respect to the constraint variables (see Lemma~\ref{subadd} below). If one can prove the strict
subadditivity inequality, all of what is proved in Section~\ref{VarProbP1} below should be readily
extendable to study the
problem (P2), in which case an analogous
result of Theorem~\ref{SOexistence} concerning the existence of minimizers will hold for the problem (P2) as well.
To obtain stability properties of solitary waves, one then considers the problem of finding for any $(l,r)\in (0, \infty)\times \mathbb{R},$
\begin{equation}\label{WdefNKK}
\Lambda_{\textrm{NKK}}(l,r)=\inf \{K(U):U=(h,g_1,g_2)\in \Y,\ Q(h)=l, \ \textrm{and}\ G(U)=r\}.
\end{equation}
The family of problems \eqref{WdefNKK} is suitable for studying the stability properties of solitary
waves for \eqref{3LSINKK} because both $K(U)$ and the constraint
functionals $Q(h)$ and $G(U)$ are invariants of motion of \eqref{3LSINKK}.
Once the existence of minimizers for the problem (P2) is guaranteed, one can follow the same arguments as
in Section~\ref{fullvar} below to obtain the stability result for the (1+2)-component NLS-gKdV solitary waves.
\end{remark}

\section{Local Well-Posedness in the Energy Regularity and Below}\label{LWP-Theory}
In this section  we supply proofs of the local-well-posedness results  to the IVPs associated to the (2+1)-component and (1+2)-component NLS-KdV systems.
We provide details of the proof of the Theorem \ref{Thm1} only, because the proof of the Theorem \ref{Thm2} follows similarly.

\subsection{Preliminary estimates}
Here we recall some important  smoothing properties related to the free propagators $S(t)=e^{it\partial_x^2}$ and $V(t)=e^{-t\partial_x^3}$, which will be useful to construct the local solutions in the energy space.

\smallskip
We begin by recalling the low-high projections operators via dyadic decomposition in the line. Let $\eta$ be a smooth nonnegative function such that
$$
\eta(\xi)=
\begin{cases}
1& \text{if}\; |\xi| \le 1,\\
0& \text{if}\; |\xi| \ge 2
\end{cases}
$$
with $\widehat{\phi}(\xi) = \eta(\xi)-\eta(2\xi) \in \mathcal{S}(\mathbb{R})$, supported in the set $\big\{\xi;\; 1/2 \le |\xi| \le 2\big\}$,  and satisfying $\displaystyle \sum\limits_{j\in \mathbb{Z}}\widehat{\phi}(\xi/2^j)=1,\; \xi \neq 0.$
Define $\phi_l$  by

$$\widehat{\phi_l}(\xi)=1- \sum\limits_{j=1}^{\infty}\widehat{\phi}\big(\tfrac{\xi}{2^j}\big)=\sum\limits_{j=0}^{\infty}\widehat{\phi}(2^j \xi)$$
and denote by $P_{\boldsymbol{l}}$  and $P_{\boldsymbol{h}}$ the operators
\begin{equation}\label{low-high-operators}
P_{\boldsymbol{l}}f = \phi_l *f\quad \text{and} \quad  P_{\boldsymbol{h}}f = (I-P_{\boldsymbol{l}})f.
\end{equation}
\begin{prop} The operators $P_{\boldsymbol{l}}f$ and $P_{\boldsymbol{h}}$ satisfy the following properties:
\medskip
\begin{enumerate}
\item[(a)] $\|P_{\boldsymbol{l}}D^{\alpha}f\|_{L^{\nu}}\lesssim \|f\|_{L^2}$\,for all $f\in \mathcal{S}(\mathbb{R})$,\;$\alpha \ge 0$ and $2\le \nu < \infty$.\medskip

\item[(b)] $\|P_{\boldsymbol{h}}g\|_{L^{\nu}_xL^{\rho}_T}\lesssim \|g\|_{L^{\nu}_xL^{\rho}_T}$\, for all $g\in \mathcal{S}(\mathbb{R}^2)$\; and\; $1\le \nu, \rho \le \infty$.
\end{enumerate}
\end{prop}
\begin{proof}
Since $\widehat{\phi_l}(\xi)=0$ for $\xi\ge 2$, the first estimate follows by using Sobolev's embedding and Plancherel's theorems. Indeed,
$$\|P_{\boldsymbol{l}}D^{\alpha}f\|_{L^{\nu}}\lesssim  \|D^{\frac12-\frac1{\nu}}P_{\boldsymbol{l}}D^{\alpha}f\|_{L^2}
=\|\widehat{\phi_l}(\xi)|\xi|^{\alpha +\frac12 -\frac1{\nu}}\hat{f}(\xi)\|_{L^2}\lesssim \|f\|_{L^2}.$$
The estimate in (b) can be found in \cite{[Mu]}.
\end{proof}

\smallskip
For any unitary group  $\big\{W(t)\big\}_{t\in \mathbb{R}}$ in $L^2(\mathbb{R})$ we have the following retarded convolution estimate
\begin{equation}\label{Estimate-Retarded-Conv}
\bigg\|\displaystyle \int_0^tW(t-t')f(\cdot, t')\,dt'\bigg\|_{L^{\infty}_TL^2_x}\le \|f\|_{L^1_TL^2_x}.
\end{equation}
Indeed, by Minkowski inequality and the properties of $W$ we get
$$\bigg\|\int_0^tW(t-t')f(\cdot, t')\,dt'\bigg\|_{L^2_x}\le \int_0^t\|f(\cdot, t')\|_{L^2_x}\,dt'.$$
Therefore,
$$\bigg\|\displaystyle \int_0^tW(t-t')f(\cdot, t')\,dt'\bigg\|_{L^{\infty}_TL^2_x} \le \sup\limits_{0\le t \le T}\int_0^t\|f(\cdot, t')\|_{L^2_x}\,dt'\le \|f\|_{L^1_TL^2_x}.$$

\medskip
Obviously, $S(t)$ and $V(t)$ satisfy \eqref{Estimate-Retarded-Conv}.

\smallskip
\begin{lem}[Smoothing effects for $e^{it\partial_x^2}$]\label{SEfeects-NLS}
Let $\mu>3/4$ and $s>1/2$. Then, for all positive $T$ we have the following  maximal function type estimates
\begin{enumerate}
\item[(a)] $\|S(t)\varphi\|_{L^2_xL^{\infty}_T} \le c(1+T)^{\mu}\|\varphi\|_{H^s}$.\medskip

\item[(b)] $\bigg\|\displaystyle \int_0^tS(t-t')f(\cdot, t')\,dt'\bigg\|_{L^2_xL^{\infty}_T}\le c(1+T)^{\mu}\|J^s_xf\|_{L^1_TL^2_x}$.
\end{enumerate}
\medskip
All constants $c$ are independent of the time $T$.
\end{lem}
\begin{proof}
For the proof of estimate in  (a) see Corollary 2.9 in \cite{KPV-a}. The estimate in (b) is obtained as follows:
\medskip
\begin{equation}\label{maximal-function-dual}
\begin{split}
\bigg\|\displaystyle \int_0^tS(t-t')f(\cdot, t')\,dt'\bigg\|_{L^2_xL^{\infty}_T}&
\le\displaystyle \int_0^T \big\|S(t)S(-t')f(\cdot, t')\big\|_{L^2_xL^{\infty}_T}\,dt'\\
&\le c\displaystyle \int_0^T (1+T)^{\mu}\big\|J^s_xS(-t')f(\cdot, t')\big\|_{L^2_x}\,dt'\\
&\le c\displaystyle \int_0^T (1+T)^{\mu}\big\|J^s_xf(\cdot, t')\big\|_{L^2_x}\,dt'\\
&= c(1+T)^{\mu}\| J^s_xf\|_{L^1_Tl^2_x},
\end{split}
\end{equation}
and we  finished the proof.
\end{proof}
The following result tells about similar results for the unitary group $V(t)$.
\begin{lem}[Smoothing effects for $e^{-t\partial_x^3}$]\label{SEfeects-KdV}
Let $\mu$\, and\, $s > 3/4$. Then, for all positive $T$  the following maximal function type estimates hold true.
\begin{enumerate}
\item[(a)] $\|V(t)\varphi\|_{L^2_xL^{\infty}_T}\le c(1+T)^{\mu}\|\varphi\|_{H^s}$. \medskip

\item[(b)] $\bigg\|\displaystyle \int_0^tV(t-t')g(\cdot, t')\,dt'\bigg\|_{L^2_xL^{\infty}_T}\le c(1+T)^{\mu}\|J^s_xg\|_{L^1_TL^2_x}$.\medskip
\end{enumerate}
Also, we have
\begin{enumerate}
\item[(c)] $\displaystyle \bigg\|\partial_x\int_0^tV(t-t')g(\cdot, t')\,dt' \bigg\|_{L^{\infty}_TL^2_x}\le  c\|g\|_{L^1_xL^2_T}$.
\end{enumerate}
\medskip
All constants $c$ are independent of the time $T$.
\end{lem}
\begin{proof}
For the proof of estimate (a) see Corollary 2.9 in \cite{KPV-a} and (b) is obtained  similarly to \eqref{maximal-function-dual}. The estimate in (c) is the classical dual version of the Kato type smoothing effect
\begin{equation}\label{Kato-SE}
\|\partial_xV(t)\varphi \|_{L^{\infty}_xL^2_t} \le c\|\varphi\|_{L^2},
\end{equation}
proved in Lemma  2.1 of \cite{KPV-a}.
\end{proof}
Another important ingredient is the Proposition 2.7 of \cite{[Mo]}, which establishes a version of the so-called Christ-Kiselev lemma. The result reads as follows:
\begin{prop}\label{Prop-Christ-Kiselev}
	Let $s_1, s_2\in \mathbb{R}$\; and\; $1\le \rho_1, \nu_1, \rho_2, \nu_2\le \infty$ such that for all $\varphi \in \mathcal{S}(\mathbb{R})$,
	\begin{align}
	&\|D^{s_1}_xV(t)\varphi\|_{L^{\nu_1}_xL^{\rho_1}_T}\lesssim \alpha_T\|\varphi\|_{L^2},\medskip \\
	&\|D^{s_2}_xV(t)\varphi\|_{L^{\nu_2}_xL^{\rho_2}_T}\lesssim \beta_T\|\varphi\|_{L^2},
	\end{align}
	where $\alpha_T,\, \beta_T$  are positive constants depending on $T$. Then for all $f \in \mathcal{S}(\mathbb{R}^2)$,
	\begin{align}
	&\bigg\|D^{s_2}_x\int_0^tV(t-t')f(x,t')\,dt' \bigg\|_{L^{\infty}_TL^2_x}\lesssim \beta_T\|f\|_{L^{\nu'_2}_xL^{\rho'_2}_T},\medskip \\
	&\bigg\|D^{s_1+s_2}_x\int_0^tV(t-t')f(x,t')\,dt'\bigg\|_{L^{\nu_1}_xL^{\rho_1}_T}\lesssim  \alpha_T\beta_T\|f\|_{L^{\nu'_2}_xL^{\rho'_2}_T},
	\end{align}
	provided the conditions
	\begin{equation}
	\min(\nu_1, \rho_1)> \max(\nu'_2, \rho'_2)\quad \text{or}\quad \rho_1=\infty\;\; \text{and}\;\; \nu'_2,\, \rho'_2 < \infty.
	\end{equation}
\end{prop}

\smallskip
The following result is concerned to the localized maximal function type estimate for $V(t)$ in high frequencies and its proof can be found in \cite{Te}. Here we sketch  the proof with only one small difference.

\begin{lem}\label{localized-maximal-function-KdV}
Let $\mu >3/4$.  Then for all $g\in \mathcal{S}(\mathbb{R}^2)$, we have
$$\bigg\|\int_0^tV(t-t')P_{\boldsymbol{h}}g(x,t')\,dt' \bigg\|_{L^2_xL^{\infty}_T}\lesssim (1+T)^{\mu}\|P_{\boldsymbol{h}}g\|_{_{L^1_xL^2_T}}.$$
\end{lem}
\begin{proof}
From the definition of $P_{\boldsymbol{h}}$ it follows that
$\displaystyle \|J^s_xP_{\boldsymbol{h}} \varphi \|_{L^2_x}\lesssim \|D^s_xP_{\boldsymbol{h}}\varphi\|_{L^2_x}$
for all $s\in \mathbb{R}$ and $\varphi \in \mathcal{S}(\mathbb{R})$. So,  taking  $s>3/4$, by Lemma \ref{SEfeects-KdV}-(a) and previous inequality we have
$$\|V(t)P_{\boldsymbol{h}} \varphi\|_{L^2_xL^{\infty}_T}\lesssim (1+T)^{\mu}\|J^s_xP_{\boldsymbol{h}}\varphi\|_{L^2_x}
\lesssim (1+T)^{\mu}\|D^s_xP_{\boldsymbol{h}}\varphi\|_{L^2_x},$$
which implies
\begin{equation}\label{L-MaxF-a}
\|D^{-1}V(t)P_{\boldsymbol{h}} \varphi\|_{L^2_xL^{\infty}_T}\lesssim (1+T)^{\mu}\|P_{\boldsymbol{h}}\varphi\|_{L^2_x},
\end{equation}
for all $\varphi \in \mathcal{S}(\mathbb{R})$.

\medskip
On the other hand, by \eqref{Kato-SE}
\begin{equation}\label{L-MaxF-b}
\|D^1_xV(t)P_{\boldsymbol{h}} \varphi\|_{L^{\infty}_xL^2_T}\lesssim \|P_{\boldsymbol{h}}\varphi\|_{L^2_x}.
\end{equation}

\medskip
Finally, applying Proposition \ref{Prop-Christ-Kiselev} with $s_1=-1$,\;  $s_2=1$,\; $(\nu_1, \rho_1)=(2,\infty)$ and $(\nu_2, \rho_2)=(\infty,2)$ we obtain the desired result.
\end{proof}

\medskip
Finally, we recall  the following particular version of the  Gagliardo-Niremberg inequality.

\begin{prop}\label{GN-inequality}Let $p\ge 2$. Then, for any $f\in H^1(\mathbb{R})$ it holds that
$$\|f\|_{L^p}\le c_p\|f\|^{\frac12 + \frac1p}_{L^2}\|\partial_x f\|^{\frac12-\frac1p}_{L^2}\le c_p\|f\|_{H^1}.$$
\end{prop}

We finish this subsection by recording some more linear and nonlinear estimates in the framework of the Bourgain space. These estimates are used to get well-posedness theory below energy space for some particular nonlinearities involved in the systems \eqref{3LSI} and \eqref{3LSINKK}. To simplify the exposition, we borrow notations from \cite{CL}.

\smallskip
Let $W_{\phi}(t):=e^{-it(-i\partial_x)}$ represents the unitary group that describes the solution of the linear problem
\begin{equation}\label{lin-eq1}
i\partial_tw -\phi(-i\partial_x)w=0.
\end{equation}
With this notation, let us define the associated function space $X^{s,b}(\phi)$, $s,b\in\R$ as the completion of the Schwartz space with respect to the norm
\begin{equation}
\|f\|_{X^{s,b}(\phi)} := \|W_{\phi}(-t)f\|_{H_t^b(\R;H_x^s(\R)} = \|\langle\xi\rangle^s\langle\tau+\phi(\xi)\rangle^b\widehat{f}(\tau,\xi)\|_{L_{\tau}^2L_{\xi}^2}.
\end{equation}
Considering respectively $\phi(\xi) =\phi_1(\xi) =\xi^2$ and $\phi(\xi) =\phi_2(\xi) =-\xi^3$ we have
$X^{s,b}(\phi_1)=X^{s,b}$ and $X^{k,b}(\phi_2)=Y^{k,b}$.

\begin{lem}\label{lem-0}
Let  $0\leq T\leq 1$ and $-\frac12<b'\leq 0\leq b\leq b'+1$. Then for any $f\in H^s$ and $g\in X^{s,b'}(\phi)$, the following estimates hold true.
\begin{equation}\label{linear -1}
\|\psi(t)W_{\phi}(t)f\|_{X^{s,b}(\phi)}\lesssim \|f\|_{H^s},
\end{equation}
\begin{equation}\label{liner -2}
\|\psi_T(t)\int_0^tW_{\phi}(t-t')g(t',\cdot)dt'\|_{X^{s,b}(\phi)}\lesssim T^{1-b+b'}\|g\|_{X^{s,b'}(\phi)}.
\end{equation}
\end{lem}
\begin{proof}
Proof of this lemma can be found in \cite{GTV}, so we omit the details.
\end{proof}

\noindent Now we state the following bilinear estimate whose proof can be found in \cite{KPV}.

\begin{lem}\label{lem-1}
Let $\kappa>-\frac34$, $b, b'=\frac12+$, then for any $v, w\in Y^{\kappa, b}$
\begin{equation}\label{eq2.6}
\|\partial_x(vw)\|_{Y^{\kappa, b'-1}}\lesssim \|v\|_{Y^{\kappa, b}}\|w\|_{Y^{\kappa,b}}.
\end{equation}
\end{lem}

\noindent Also, we record the following trilinear estimate from \cite{BOP}.

\begin{lem}\label{lem-1.1}
Let $s\geq 0$, $\frac12<b<1$ and $u\in H^{s,b}$. Then for any $a\geq 0$, one has
\begin{equation}\label{trilinear}
\||u\|^2u\|_{H^{s, -a}}\leq C\|u\|_{X^{s,b}}^3.
\end{equation}
\end{lem}

\noindent The following results are obtained in \cite{WU}.

\begin{lem}\label{lem-2}
Let $\kappa \geq -1$ and suppose $s-\kappa\leq 2$ when $s\geq 0$,  $s+\kappa \geq -2$ when $s<0$. Also consider $c, c', b =\frac12+$.
Then for any $u\in X^{s,c}$ and $v\in Y^{\kappa, b}$, one has
\begin{equation}\label{eq2.7}
\|uv\|_{X^{s, c'-1}}\lesssim \|u\|_{X^{s,c}}\|v\|_{Y^{\kappa, b}}.
\end{equation}
\end{lem}

\begin{lem}\label{lem-3}
Let $s>-\frac14$, $\kappa <4s$ and consider $\kappa-s<1$ when $s\geq 0$, $\kappa-2s<1$ when $s<0$. Also, take $b', c=\frac12+$. Then for any $u_1, u_2\in X^{s,c}$
\begin{equation}\label{eq2.8}
\|\partial_x(u_1\bar{u}_2)\|_{Y^{\kappa, b'-1}}\lesssim\|u_1\|_{X^{s,c}}\|u_2\|_{X^{s,c}}.
\end{equation}
\end{lem}

\subsection{Local and global theory for (2+1)-component NLS-gKdV}
Throughout this section we assume the following conditions on the indices of the nonlinearities
\begin{equation}\label{3LSI-parameter-conditions}
q_j>0\; (j=1,2) \quad \text{and} \quad p=\tfrac{n_1}{n_2} \ge 1\;\,  \text{with}\;\,  n_1, n_2\in \mathbb{N}\;\, \text{and}\;\, n_2\;\, \text{odd}.
\end{equation}

Using Duhamel's formula, we consider the IVP associated to the system \eqref{3LSI} in the equivalent system of integral equations
\begin{equation}\label{equiv-syst-a}
\begin{cases}
u_j(t)=S(t){u_j}_0+i\displaystyle \int_0^tS(t-t')\left[\alpha_j u_jv+\gamma_j|u_j|^{q_j}u\right](t')dt'\;\; (j=1,2),\medskip\\
v(t) = V(t)v_0 -\displaystyle\int_0^t V(t-t')\left[\beta v^p\partial_xv +\partial_x\big(\alpha_1|u_1|^2+\alpha_2|u_2|^2\big) \right](t')dt'.
\end{cases}
\end{equation}

Our goal is to solve \eqref{equiv-syst-a} by applying the contraction mapping principle in a suitable subspace of the continuous functions
$\mathcal{C}\big( [0, T];\; H^1 \times H^1 \times H^1\big)$. The  main tool in the proof is the use of a localized maximal function, described in Lemma \ref{localized-maximal-function-KdV}, in order to estimate the non-homogeneous term
$$\displaystyle\int_0^t V(t-t')g(x,t')dt',$$ with a good control for  the terms $\partial_x(|u_j|^2)$.

\begin{remark}
	We note that our proof include and extend  the results given by Guo and Miao in \cite{Guo}, for energy regularity, in the context of the classical (1+1)-component NLS-KdV system ($p=1$ and $\gamma_j=0$). In fact, here we contemplate fractional powers for KdV component and also we emphasize that our fixed-point procedure is developed in a  different
	environment than the one used in  \cite{Guo}.
\end{remark}

\begin{proof}[\bf{Proof of Theorem \ref{Thm1} (local theory)}]
Let $\mu >  3/4$.  For positive numbers  $T$,  $M_1, M_2$ and $M$, to be chosen later,  we define a function space
$$\mathcal{Z}_T=\Big\{(u_1, u_2, v)\in \mathcal{C}\big([0,T];\, H^1\times H^1\times H^1 \big);\;  \|u_j\|_{\mathcal{S}_T}\le M_j \; \text{and}\; \|v\|_{\mathcal{K}_T} \le M\Big\},$$
where
\begin{equation}
\|u_j\|_{\mathcal{S}_T}:= \|J^1_xu_j\|_{L^{\infty}_TL^2_x} + (1+T)^{-\mu}\|u_j\|_{{L^2_xL^{\infty}_T}},\quad j=1,2,
\end{equation}
and
\begin{equation}
\|v\|_{\mathcal{K}_T}:= \|J^1_xv\|_{L^{\infty}_TL^2_x} + (1+T)^{-\mu}\|v\|_{{L^2_xL^{\infty}_T}}.
\end{equation}
We equip  the space $\mathcal{Z}_T$ with  the norm defined by
\begin{equation}\label{Z-norm}
\|(u_1, u_2, v)\|_{\mathcal{Z}_T}:= \|u_1\|_{\mathcal{S}_T} + \|u_2\|_{\mathcal{S}_T}+ \|v\|_{\mathcal{K}_T}.
\end{equation}

\noindent It is easy to check that $\mathcal{Z}_T$ is a complete metric space with respect to the norm defined in \eqref{Z-norm}.

\smallskip

\noindent For $\Delta:=(u_1, u_2, v)\in \mathcal{Z}_T$ and with $j=1,2$, we define the operators
\begin{equation}\label{Appl-1}
\begin{cases}
\Phi_j(\Delta) = S(t){u_j}_{0}+i\displaystyle \int_0^tS(t-t')\left[\alpha_j (u_jv)+\gamma_j|u_j|^{q_j}u_j\right]\,dt',\medskip\\
\Psi(\Delta) = V(t)v_0 -\!\!\displaystyle\int_0^t V(t-t')\left[\beta v^p\partial_xv +\partial_x\big(\alpha_1|u_1|^2+\alpha_2|u_2|^2\big)\right]\,dt'.
\end{cases}
\end{equation}

{\noindent \underline{Claim:}}  $\big(\Phi_1, \Phi_2, \Psi\big)(\mathcal{Z}_T) \subset \mathcal{Z}_T$ for a suitable positive time $T$.

\smallskip

In what follows we consider $\Delta \in \mathcal{Z}_T$.

\smallskip

\noindent{\bf{Estimates for $\boldsymbol \Phi_j$}.} Using \eqref{Estimate-Retarded-Conv},  Lemma \ref{SEfeects-NLS} with
$s=1$, the algebra structure of $H^1(\mathbb{R})$ and  H\"older's inequality it follows that
\begin{equation}\label{lwp-Estimate-NLS-a}
\begin{split}
\|\Phi_j(\Delta)\|_{\mathcal{S}_T}& =\|J^1_x\Phi_j(\Delta)\|_{L^{\infty}_TL^2_x} + (1+T)^{-\mu}\|\Phi_j(\Delta)\|_{L^2_xL^{\infty}_T}\\
& \lesssim  \|{u_j}_0\|_{H^1}+ \|J^1_x(u_jv)\|_{L^1_TL^2_x} + \|J^1_x(|u_j|^{q_j}u_j)\big \|_{L^1_TL^2_x}\\
& \lesssim \|{u_j}_0\|_{H^1} + T\,\|u_j\|_{L^{\infty}_TH^1_x}\|v\|_{L^{\infty}_TH^1_x}
+ \|J^1_x(|u_j|^{q_j}u_j)\big \|_{L^1_TL^2_x}.
\end{split}
\end{equation}

On the other hand, using the Gagliardo-Nirenberg inequality from Lemma \ref{GN-inequality}, one can obtain
\begin{equation}\label{lwp-Estimate-NLS-b}
\begin{split}
 \|J^1_x(|u_j|^{q_j}u_j)\big \|_{L^1_TL^2_x}&= \int_0^T\big \|\,|u_j|^{q_j}u_j \big \|_{L^2_x}\,dt + \int_0^T\big \|\partial_x (|u_j|^{q_j}u_j)\big \|_{L^2_x}\,dt\\
& \le \int_0^T\|u_j\|^{q_j+1}_{L^{2(q_j+1)}_x}\,dt+ (q_j+1)\int_0^T\|\,|u_j|^{q_j}\partial_xu_j\|_{L^2_x}\,dt\\
&\lesssim  T\,\|u_j\|^{q_j+1}_{L^{\infty}_TH^1_x} + (q_j+1)\int_0^T \|u_j\|^{q_j}_{L^{\infty}_x}\|\partial_xu_j\|_{L^2_x}\,dt\\
&\lesssim T\,\|u_j\|^{q_j+1}_{L^{\infty}_TH^1_x}.
\end{split}
\end{equation}
Thus, combining \eqref{lwp-Estimate-NLS-a} and \eqref{lwp-Estimate-NLS-b} there are positive constants $C_0$ and $C_{j}$ such that
\begin{equation}\label{lwp-Estimate-NLS-c}
\begin{split}
\|\Phi_j(\Delta)\|_{\mathcal{S}_T}&\le C_0\|{u_j}_0\|_{H^1} + C_j T\,\Big(\|u_j\|_{L^{\infty}_TH^1_x}\|v\|_{L^{\infty}_TH^1_x} + \|u_j\|^{q_j+1}_{L^{\infty}_TH^1_x}\Big)\\
&\le  C_0\|{u_j}_0\|_{H^1} + C_j\,T\,\Big(\|u_j\|_{\mathcal{S}_T}\|v\|_{\mathcal{K}_T} + \|u_j\|_{\mathcal{S}_T}^{q_j+1}\Big)\\
&\le  C_0\|{u_j}_0\|_{H^1} + C_j\,T\,\big(M_jM + M_j^{q_j+1}\big),
\end{split}
\end{equation}
with $C_j$ depending on the parameters $\alpha_j, \gamma_j$ and $q_j$.

\medskip

\noindent{\bf{Estimates for $\boldsymbol \Psi$}.} The estimate \eqref{Estimate-Retarded-Conv} combined with the H\"older's inequality and the Sobolev embedding yields
\begin{equation}\label{lwp-Estimate-KdV-a}
\begin{split}
\|\Psi(\Delta)\|_{L^{\infty}_TL^2_x}& \lesssim  \|v_0\|_{L^2} + \big\|\partial_x(|u_1|^2) + \partial_x(|u_2|^2) + v^p\partial_x v \big\|_{L^1_TL^2_x}\\
&\lesssim \|v_0\|_{L^2} + T\, \Big( \sum\limits_{j=1}^{2}\|u_j\|_{L^{\infty}_TL^{\infty}_x} \|\partial_x u_j\|_{L^{\infty}_TL^2_x}+
\|v^p\|_{L^{\infty}_TL^{\infty}_x} \|\partial_x v\|_{L^{\infty}_TL^2_x}\Big)\\
& \lesssim \|v_0\|_{L^2} + T\, \Big( \|u_1\|^2_{L^{\infty}_TH^1_x}+ \|u_2\|^2_{L^{\infty}_TH^1_x} +\|v\|^{p+1}_{L^{\infty}_TH^1_x} \Big)\\
&\lesssim\|v_0\|_{L^2} + T\, \Big(\|u_1\|^2_{\mathcal{S}_T} + \|u_2\|^2_{\mathcal{S}_T} + \|v\|^{p+1}_{\mathcal{K}_T}\Big)
\end{split}
\end{equation}

In order to estimate the derivative of $\Psi(\vec{w})$ we use Lemma \ref{SEfeects-KdV}-(c), to obtain
\begin{equation}\label{lwp-Estimate-KdV-b}
\big\|\partial_x\Psi(\Delta)\big\|_{L^{\infty}_TL^2_x}\lesssim \|v_0\|_{H^1} +  \big\|v^p\partial_x v+ \partial_x(|u_1|^2) +\partial_x(|u_2|^2)\big\|_{L^1_xL^2_T}.
\end{equation}
\medskip
Furthermore, using H\"older's inequality, Fubini's Theorem and Sobolev embedding, we get
\begin{equation}\label{lwp-Estimate-KdV-c}
\begin{split}
\big\|v^p\partial_x v\big\|_{L^1_xL^2_T}&= \|v v^{p-1}\partial_xv\|_{L^1_xL^2_T}\\
&\le \|v\|_{L^2_xL^{\infty}_T} \|v^{p-1}\partial_xv\|_{L^2_xL^2_T}\\
&\le \|v\|_{L^2_xL^{\infty}_T} \|v^{p-1}\|_{L^{\infty}_TL^{\infty}_x}\|\partial_xv\|_{L^2_TL^2_x}\\
&\le T^{1/2}\|v\|_{L^2_xL^{\infty}_T} \|v^{p-1}\|_{L^{\infty}_TL^{\infty}_x}\|\partial_xv\|_{L^{\infty}_TL^2_x}\\
&\le T^{1/2}\|v\|_{L^2_xL^{\infty}_T}\|v\|^p_{L^{\infty}_TH^1_x}
\end{split}
\end{equation}
and
\begin{equation}\label{lwp-Estimate-KdV-d}
\begin{split}
\big\|\partial_x |u_j|^2\big\|_{L^1_xL^2_T}& \le  \|\bar{u_j}\partial_xu_j\|_{L^1_xL^2_T} + \|u\partial_x\bar{u_j}\|_{L^1_xL^2_T}\\
&\le 2 \|u_j\|_{L^2_xL^{\infty}_T} \|\partial_xu_j\|_{L^2_xL^2_T}\\
&\le 2 T^{1/2}\|u_j\|_{L^2_xL^{\infty}_T}\|u_j\|_{L^{\infty}_TH^1_x}.
\end{split}
\end{equation}

Combining \eqref{lwp-Estimate-KdV-b}, \eqref{lwp-Estimate-KdV-c} and \eqref{lwp-Estimate-KdV-d}, we have
\begin{equation}\label{lwp-Estimate-KdV-e}
\big\|\partial_x\Psi(\Delta)\big\|_{L^{\infty}_TL^2_x}
\lesssim \|v_0\|_{H^1} +  T^{1/2}\Big(\|u_1\|^2_{\mathcal{S}_T} + \|u_2\|^2_{\mathcal{S}_T} + \|v\|^{p+1}_{\mathcal{K}_T} \Big).\\
\end{equation}	
		
Now, we proceed to estimate the more delicate term, the maximal function $\big\|\Psi(\Delta)\big\|_{L^2_xL^{\infty}_T}$. For this purpose we put
$$g(\Delta)=\beta v^p\partial_x v+ \alpha_1\partial_x(|u_1|^2) + \alpha_2\partial_x(|u_2|^2.)$$
From Lemma \ref{SEfeects-KdV}-(a) we have that
\begin{equation}\label{lwp-Estimate-KdV-f}
\begin{split}
(1+T)^{-\mu}\big\|\Psi(\Delta)\big\|_{L^2_xL^{\infty}_T}&\lesssim \|v_0\|_{H^1} +
(1+T)^{-\mu}\bigg\|\displaystyle \int_0^tV(t-t')g(\Delta)\,dt'\bigg\|_{L^2_xL^{\infty}_T}\\
&=\|v_0\|_{H^1} + (1+T)^{-\mu}G(\Delta).
\end{split}
\end{equation}
Splitting $g$ by the projection operators and using the Proposition \ref{low-high-operators} and the  Lemma \ref{localized-maximal-function-KdV}, we get
\begin{equation}\label{lwp-Estimate-KdV-g}
\begin{split}
G(\Delta)&\le \bigg\|\displaystyle \int_0^tV(t-t')P_{\boldsymbol{l}}g(\Delta)\,dt'\bigg\|_{L^2_xL^{\infty}_T}
 + \quad  \bigg\|\displaystyle \int_0^tV(t-t')P_{\boldsymbol{h}}g(\Delta)\,dt'\bigg\|_{L^2_xL^{\infty}_T}\\
 &\lesssim \|J^1_xP_{\boldsymbol{l}}g(\Delta)\|_{_{L^1_TL^2_x}} +(1+T)^{\mu}\|P_{\boldsymbol{h}}g(\Delta)\|_{_{L^1_xL^2_T}}\\
 &\lesssim  \|g(\Delta)\|_{_{L^1_TL^2_x}} + (1+T)^{\mu}\|g(\Delta)\|_{_{L^1_xL^2_T}}.
\end{split}
\end{equation}

On the other hand
\begin{equation*}\label{lwp-Estimate-KdV-h}
\begin{split}
\|g(\Delta)\|_{_{L^1_TL^2_x}}&\le \big\|\beta v^p\partial_x v+ \alpha_1\partial_x(|u_1|^2) + \alpha_2\partial_x(|u_2|^2)\big\|_{_{L^1_TL^2_x}}\\
&\lesssim T\,\bigg(\, \|v^p\|_{L^{\infty}_TL^{\infty}_x} \|\partial_xv\|_{L^{\infty}_TL^2_x} +
\sum\limits_{j=1}^{2}\|u_j\|_{L^{\infty}_TL^{\infty}_x} \|\partial_xu_j\|_{L^{\infty}_TL^2_x} \,\bigg)\\
&\lesssim T\,\bigg( \|v\|^{p+1}_{L^{\infty}_TH^1_x} + \|u_1\|^2_{L^{\infty}_TH^1_x}+ \|u_2\|^2_{L^{\infty}_TH^1_x}\bigg)\\
&\lesssim  T\, \Big(\|u_1\|^2_{\mathcal{S}_T} + \|u_2\|^2_{\mathcal{S}_T} + \|v\|^{p+1}_{\mathcal{K}_T}\Big).
\end{split}
\end{equation*}
The same arguments as in \eqref{lwp-Estimate-KdV-c}-\eqref{lwp-Estimate-KdV-d} yield
\begin{equation*}\label{lwp-Estimate-KdV-i}
\|g(\Delta)\|_{_{L^1_xL^2_T}}\le T^{1/2}\, \Big(\|u_1\|^2_{\mathcal{S}_T} + \|u_2\|^2_{\mathcal{S}_T} + \|v\|^{p+1}_{\mathcal{K}_T} \Big).
\end{equation*}
Hence,
\begin{equation}\label{lwp-Estimate-KdV-j}
\|G(\Delta)\|_{_{L^1_xL^2_T}}\le \Big(T+(1+T)^{\mu}T^{1/2}\Big)\, \Big(\|u_1\|^2_{\mathcal{S}_T} + \|u_2\|^2_{\mathcal{S}_T} + \|v\|^{p+1}_{\mathcal{K}_T} \Big).
\end{equation}

Combining \eqref{lwp-Estimate-KdV-f}, \eqref{lwp-Estimate-KdV-g} and \eqref{lwp-Estimate-KdV-j}, one obtains
\begin{equation}\label{lwp-Estimate-KdV-k}
\begin{split}
(1+T)^{-\mu}\big\|\Psi(\Delta)\big\|_{L^2_xL^{\infty}_T}&\lesssim \|v_0\|_{H^1} + \big(T +T^{1/2}\big)\, \Big(\|u_1\|^2_{\mathcal{S}_T} + \|u_2\|^2_{\mathcal{S}_T} + \|v\|^{p+1}_{\mathcal{K}_T} \Big).
\end{split}
\end{equation}

\medskip
From \eqref{lwp-Estimate-KdV-a}, \eqref{lwp-Estimate-KdV-e} and \eqref{lwp-Estimate-KdV-k} there are  positive constants $\tilde{C}_0$ and $\tilde{C}$, where $\tilde{C}$ depends on the parameters $\beta$ and $\alpha_j$ such that
\begin{equation}\label{lwp-Estimate-KdV-l}
\begin{split}
\|\Psi(\Delta)\|_{\mathcal{K}_T} &\le \tilde{C}_0\|v_0\|_{H^1} + \tilde{C}\big(T +T^{1/2}\big)\, \Big(\|u_1\|^2_{\mathcal{S}_T} + \|u_2\|^2_{\mathcal{S}_T} + \|v\|^{p+1}_{\mathcal{K}_T} \Big)\\
&\le \tilde{C}_0\|v_0\|_{H^1} + \tilde{C}\big(T +T^{1/2}\big)\, \Big(M_1^2 + M_2^2 + M^{p+1}\Big)
\end{split}
\end{equation}

\medskip
Now we put
\begin{equation}
M_j=2C_0\|{u_j}_0\|_{H^1}\quad  \text{and} \quad M=2\tilde{C}_0\|v_0\|_{H^1}
\end{equation}
and choose $T>0$ satisfying the conditions
\begin{equation}\label{Time-Existence}
\begin{cases}
T \le  \frac{M_j}{2C_j}\big(M_jM + M_j^{q_j+1}\big)^{-1},\; j=1,2,\medskip\\
T +T^{1/2} \le \frac{M}{2\tilde{C}}\Big(M_1^2 + M_2^2 + M^{p+1}\Big)^{-1},
\end{cases}
\end{equation}
which guarantee the inclusion $\big(\Phi_1, \Phi_2, \Psi\big)(\mathcal{Z}_T) \subset \mathcal{Z}_T$.

\medskip
 In an analogous manner, it is not difficult to prove that the application $(\Phi_1, \Phi_2, \Psi)$ is a contraction on $\mathcal{Z}_T$. Therefore, from contraction mapping principle we conclude that $(\Phi_1, \Phi_2, \Psi)$ has a unique fixed point in $\mathcal{Z}_T$ that is the solution to the integral equation \eqref{Appl-1} in the time interval $[0, T]$, with $T$
 maybe less than the one chosen in \eqref{Time-Existence}. The rest of the proof follows standard arguments, so we omit the details.
\end{proof}

\begin{proof}[\bf{Proof of Theorem \ref{Thm1} (global theory)}] Let  $1\leq p<\frac43$ and $q_j$ satisfying

\begin{equation}
0< q_j <
\begin{cases}
\;4       & \text{if}\quad \tau_j > 0,\\
\infty  & \text{if}\quad \tau_j \le 0,
\end{cases}	
\end{equation}
for $j=1,2.$ For the system \eqref{3LSI}, as recorded in \eqref{con-1}, \eqref{con-2} and \eqref{Edef}, we have the following conserved quantities
\begin{equation}\label{con-1j}
Q(u_j)(t) =\|u_j(t)\|_{L^2}^2 = \|{u_j}_0\|_{L^2}^2=:{Q_j}_0,\;\;  j=1,2,
\end{equation}
\begin{equation}\label{con-21}
H(\Delta)(t) =H(\Delta_0)=:H_0
\end{equation}
and
\begin{equation}\label{con-31}
E(\Delta)(t) =E(\Delta_0)=:E_0,
\end{equation}
where  $\Delta=(u_1, u_2, v)$ and $\Delta_0=({u_1}_0, {u_2}_0, v_0)$.

We use these conserved quantities to obtain an {\em a priori} estimate in the energy space $H^1\times H^1\times H^1$.
From \eqref{con-31}, using the definition of $E$ in \eqref{Edef}, we obtain
\begin{equation}\label{con-31a}
E_0 = \sum_{j=1}^2\|\partial_xu_j\|_{L^2_x}^2+\|\partial_xv\|_{L^2_x}^2 -\tau \int v^{p+2}dx -\sum_{j=1}^2\int \Big[\alpha_jv |u_j|^2dx +\tau_j|u_j|^{q_j+2}\Big]dx.
\end{equation}

Using the triangle and Cauchy-Schwartz inequalities, we get
\begin{equation}\label{con-31b}
\begin{split}
y(t)&:=\sum_{j=1}^2\|\partial_xu_j\|_{L^2_x}^2  +\|\partial_xv\|_{L^2_x}^2\\
&=E_0+ \tau \int v^{p+2}dx + \sum_{j=1}^2\int \alpha_j v |u_j|^2dx +\sum_{j=1}^2\tau_j\|u_j\|_{L^{q_j+2}_x}^{q_j+2}.
\end{split}
\end{equation}

Notice that if  $\tau_j\le 0$ for some $j=1,2$, then the corresponding term $\displaystyle \tau_j\|u_j\|_{L^{q_j+2}_x}^{q_j+2}$  is negative. In this case  we need not estimate this term, because its contribution does not increase the right hand  of \eqref{con-31b}. So, we only need  to consider the condition $0< q_j< \infty$, coming from the local theory.  Hence, we just consider the case $\tau_j>0$ for $j=1,2$.  From the definition of $y(t)$, we have
\begin{equation}\label{con-31c}
\begin{split}
y(t)&\leq |E_0| +|\tau|\,\|v\|_{L^{p+2}_x}^{p+2}+\sum_{j=1}^2|\alpha_j|\,\big\|v\, |u_j|^2\big\|_{L^1_x}+ \sum_{j=1}^2\tau_j\|u_j\|_{L^{q_j+2}_x}^{q_j+2}\\
&\lesssim|E_0| +\|v\|_{L^{p+2}_x}^{p+2} + \sum_{j=1}^2\|v\|_{L^2_x} \|u_j\|_{L^4_x}^2+\sum_{j=1}^2\|u_j\|_{L^{q_j+2}_x}^{q_j+2}.
\end{split}
\end{equation}

An use of Young's inequality in \eqref{con-31c}, yields
\begin{equation}\label{con-31d}
y(t)\leq C\bigg(|E_0| + \|v\|_{L^{p+2}_x}^{p+2} + \|v\|_{L^2_x}^2 + \sum_{j=1}^2\|u_j\|_{L^4_x}^4 + \sum_{j=1}^2\|u_j\|_{L^{q_j+2}_x}^{q_j+2}\bigg).
\end{equation}

Now, we use Gagliardo-Nirenberg and Young's inequalities (taking into account that the final exponent  for  $\|\partial_x v\|_{L^2_x}$ cannot exceed 2), to obtain
\begin{equation}\label{con-4a}
\|v\|_{L^{p+2}}^{p+2}\lesssim \|\partial_xv\|_{L^2}^{\frac{p}2}\|v\|_{L^2}^{\frac{p+4}2} \le
\frac1{2C} \|\partial_xv\|_{L^2}^2 + C_1\|v\|_{L^2}^{2\frac{4+p}{4-p}},
\end{equation}
where appears the restriction $0 < p <4$.

Similarly, we can obtain
\begin{equation}\label{con-4b}
\|u_j\|_{L^{q_j+2}}^{q_j+2}\leq \frac{1}{16C}\|\partial_xu_j\|_{L^2}^2 + C_2\|u_j\|_{L^2}^{2\frac{4+q_j}{4-q_j}}
= \frac{1}{16C}\|\partial_xu_j\|_{L^2}^2 + C_2{Q_j}_0^{\frac{4+q_j}{4-q_j}},
\end{equation}
with $0 < q_j <4$, and also
\begin{equation}\label{con-4c}
\begin{split}
\|u_j\|_{L^4}^4&\lesssim \|\partial_xu_j\|_{L^2}\|u_j\|_{L^2}^3\\
&\leq \frac{1}{16C}\|\partial_xu_j\|_{L^2}^2 + C_3\|u_j\|_{L^2_x}^6\\
&=\frac{1}{16C}\|\partial_xu_j\|_{L^2}^2 + C_3{Q_j}^3_0,
\end{split}
\end{equation}
for $j=1,2$.

On the other hand, from \eqref{con-21}, using definition of $H(\Delta)$ in \eqref{con-2}, one gets after applying Triangle and Cauchy-Schwartz inequalities that
\begin{equation}\label{con-21a}
\begin{split}
\|v\|_{L^2_x}^2&\leq |H_0|+2 \sum_{j=1}^2\|u_j\|_{L^2_x}\|\partial_xu_j\|_{L^2_x}\\
&\leq |H_0|+ \sum_{j=1}^2\Big[\frac{1}{16C}\|\partial_xu_j\|_{L^2_x}^2 + C_4\|u_j\|_{L^2_x}^2 \Big]\\
&=|H_0|+ \sum_{j=1}^2\Big[\frac{1}{16C}\|\partial_xu_j\|_{L^2}^2 + C_4{Q_j}_0\Big].
\end{split}
\end{equation}

Inserting the first inequality in \eqref{con-21a} into \eqref{con-4a}, one gets
\begin{equation}\label{con-4aa}
\begin{split}
\|v\|_{L^{p+2}}^{p+2}
&\leq \frac{1}{2C}\|\partial_xv\|_{L^2}^2 + C_1 \bigg(|H_0|+2 \sum_{j=1}^2\|u_j\|_{L^2_x}\|\partial_xu_j\|_{L^2_x}\bigg)^{\frac{4+p}{4-p}}\\
&\leq \frac{1}{2C}\|\partial_xv\|_{L^2}^2  + C_5|H_0|^{\frac{4+p}{4-p}}+ C_5\sum_{j=1}^2\|u_j\|_{L^2}^{\frac{4+p}{4-p}}\|\partial_xu_j\|_{L^2}^{\frac{4+p}{4-p}}\\
&\leq \frac{1}{2C}\|\partial_xv\|_{L^2}^2 + C_5|H_0|^{\frac{4+p}{4-p}}+ \sum_{j=1}^2\bigg[C_6\|u_j\|_{L^2}^{2\frac{4+p}{4-3p}}+\frac{1}{16C}\|\partial_xu_j\|_{L^2}^{2}\bigg]\\
&=\frac{1}{2C}\|\partial_xv\|_{L^2}^2 +C_5|H_0|^{\frac{4+p}{4-p}}+\sum_{j=1}^2\bigg[C_6{Q_j}_0^{\frac{4+p}{4-3p}}+\frac{1}{16C}\|\partial_xu_j\|_{L^2}^{2}\bigg],
\end{split}
\end{equation}
where we have used Young's inequality in the  third estimate  with the condition  $2\frac{4-p}{4+p}>1$, and this implies   $0< p < 4/3$. This condition combined with the restriction of the local theory gives the necessary restriction  $1\le p < 4/3$.

Now, we use \eqref{con-4b}, \eqref{con-4c} and \eqref{con-4aa} in \eqref{con-31d} and \eqref{con-21a}, to obtain
\begin{equation*}\label{con-5}
 y(t)+\|v\|^2_{L^2}\leq \frac12 y(t) + C_7\bigg( |E_0| +|H_0|+|H_0|^{\frac{4+p}{4-p}} +\sum_{j=1}^2\Big[ {Q_j}_0 +{Q_j}_0^3+{Q_j}_0^{\frac{4+q_j}{4-q_j}}+{Q_j}_0^{\frac{4+p}{4-3p}} \Big] \bigg)
\end{equation*}
and consequently,
\begin{equation}\label{con-5a}
y(t)+\|v\|^2_{L^2}\leq \tilde{C}\bigg( |E_0| +|H_0|+|H_0|^{\frac{4+p}{4-p}} +\sum_{j=1}^2\Big[ {Q_j}_0 +{Q_j}_0^3+{Q_j}_0^{\frac{4+q_j}{4-q_j}}+{Q_j}_0^{\frac{4+p}{4-3p}} \Big] \bigg)
\end{equation}

From \eqref{con-5a}, using similar estimates to those used in the previous process and \eqref{con-1j},  we obtain the following {\em a priori} estimate:
\begin{equation}\label{ap-1}
\|u_1(t)\|_{H^1}^2+\|u_2(t)\|_{H^1}^2+\|v(t)\|_{H^1}^2\leq f(\|{u_1}_0\|_{H^1}, \|{u_2}_0\|_{H^1},\|v_0\|_{H^1}).
\end{equation}

The estimate \eqref{ap-1} can be used to iterate the local existence argument to a prove the existence of the solution in any time interval $[0,T]$, for arbitrary $T>0$.
\end{proof}
\subsection{Local and global theory for the (1+2)-component NLS-gKdV}

In this subsection, we provide a proof of the well-posedness result for the (1+2)-component NLS-gKdV system \eqref{3LSINKK}. Here, we assume
\begin{equation}\label{3LSINKK-parameter-conditions}
q>0 \quad \text{and} \quad p_j=\tfrac{n_1}{n_2} \ge 1\;\,  \text{with}\;\,  n_1, n_2\in \mathbb{N}\;\, \text{and}\;\, n_2\;\, \text{odd}, \; (j=1,2).
\end{equation}

\begin{proof}[\bf{Proof of Theorem \ref{Thm2}}] The proof of  Theorem \ref{Thm2} is very similar to that of Theorem~\ref{Thm1}. In  this case, to obtain the local well-posedness result  in the framework of the proof of Theorem~\ref{Thm1},  we define
\begin{equation}
M=2C_0\|{u}_0\|_{H^1}\quad  \text{and} \quad M_j=2\tilde{C}_0\|{v_j}_0\|_{H^1}, \; j=1,2
\end{equation}
and choose $T>0$ satisfying the conditions
\begin{equation}\label{Time-Existence-2}
\begin{cases}
T \le  \frac{M}{2C_1}\big(M(M_1+M_2) + M^{q+1}\big)^{-1},\medskip\\
T +T^{1/2} \le \frac{M_j}{2\tilde{C_j}}\Big(M^2 + M_j^{p_j+1}\Big)^{-1},
\end{cases}
\end{equation}
to  perform the contraction mapping argument in  a function space
$$\mathcal{Z}_T=\Big\{(u, v_1, v_2)\in C\big([0,T];\, H^1\times H^1\times H^1 \big);\;  \|u\|_{\mathcal{S}_T}\le M \; \text{and}\; \|v_j\|_{\mathcal{K}_T} \le M_j\Big\},$$
where
\begin{equation}
\|u\|_{\mathcal{S}_T}:= \|J^1_xu\|_{L^{\infty}_TL^2_x} + (1+T)^{-\mu}\|u\|_{{L^2_xL^{\infty}_T}}\quad j=1,2,
\end{equation}
\begin{equation}
\|v_j\|_{\mathcal{K}_T}:= \|J^1_xv_j\|_{L^{\infty}_TL^2_x} + (1+T)^{-\mu}\|v_j\|_{{L^2_xL^{\infty}_T}},
\end{equation}
and the norm on $\mathcal{Z}_T$ is defined by
\begin{equation}\label{Z-norm-2}
\|(u, v_1, v_2)\|_{\mathcal{Z}_T}:= \|u\|_{\mathcal{S}_T} + \|v_1\|_{\mathcal{K}_T}+ \|v_2\|_{\mathcal{K}_T}.
\end{equation}

The proof of the global well-posedness result is also similar to that of Theorem \ref{Thm1}. Considering,
$0< q<4$ and $1\leq p_j<\frac43$, $j=1,2$, one can use the conserved quantities stated in \eqref{EdefNKK}, \eqref{con-1NKK} and \eqref{con-2NKK}, to get an {\em a priori} estimate
\begin{equation}\label{ap-1NKK}
\|u(t)\|_{H^1}^2+\|v_1(t)\|_{H^1}^2+\|v_2(t)\|_{H^1}^2\leq f(\|{u}_0\|_{H^1}, \|{v_1}_0\|_{H^1},\|{v_2}_0\|_{H^1}),
\end{equation}
which can be used to extend the local solution to any arbitrary time interval $[0,T]$, $T>0$.
So, we omit the details.
\end{proof}

\subsection{Local theory below the energy space}

The proof of the local well-posedness results stated in Theorems \ref{Thm1.1} and \ref{Thm1.2} follow with an standard contraction mapping argument in the Bourgain's space framework.  The main ingredients are the linear, bilinear and trilinear estimates recorded in Lemmas \ref{lem-0}--\ref{lem-3}.

In fact, to prove Theorem \ref{Thm1.1}, we define
\begin{equation}\label{appl-1.1}
\begin{cases}
\Phi_j(\Delta)=\psi(t)S(t){u_j}_{0}+i\psi_T(t)\displaystyle \int_0^tS(t-t')\left[\alpha_1 (u_jv)+\gamma_j|u_j|^2u_j\right](t')dt' \medskip\\
\Psi(\Delta) = \psi(t)V(t)v_0 -\psi_T(t)\displaystyle \int_0^t V(t-t')\partial_x\left[\beta v^2+\alpha_1|u_1|^2+\alpha_2|u_2|^2\right](t')dt'.
\end{cases}
\end{equation}
Taking $M_j=2C_0\|{u_j}_0\|_{H^s}$, $j=1,2$, $M =2C_0\|v_0\|_{H^k}$, and for $0<\epsilon\ll 1$
\begin{equation}\label{T-epsilon}
T^{\epsilon}\leq \frac12 \min \Big\{\frac1{C_1[M+M_1^2]}, \; \frac1{C_2[M+M_2^2]}, \; \frac{M}{C[M^2+M_1^2+M_2^2]} \Big\},
\end{equation}
one can show that  $(\Phi_1, \Phi_2, \Psi)$ is a contraction mapping on
\begin{equation*}\label{f-s}
\mathcal{Z}^{\epsilon}:= \Big\{ (u_1, u_2, v)\in X^{s, \frac12+\epsilon}\times X^{s, \frac12+\epsilon}\times Y^{k, \frac12+\epsilon} : \|u_j\|_{X^{s, \frac12+\epsilon}}\leq M_j, \; \|v\|_{Y^{k, \frac12+\epsilon}} \leq M \Big\},
\end{equation*}
with a norm defined by
\begin{equation*}
\|(u_1, u_2, v)\|_{\mathcal{Z}^{\epsilon}}:= \|u_1\|_{X^{s, \frac12+\epsilon}}+\|u_2\|_{X^{s, \frac12+\epsilon}}+\|v\|_{Y^{k, \frac12+\epsilon}}.
\end{equation*}
This process is now a classical one, so we omit the details.

The proof of Theorem \ref{Thm1.2} follows analogously.\\

 At this point, a natural question arises about global well-posedness below energy spaces. For this, the recently introduced
{\em{I-method}} \cite{CKSTT-1, CKSTT-2} combined with almost conserved quantities could be useful. This will be addressed elsewhere.
\section{Existence of Prescribed $L^2$-Norm Solutions}\label{VarProbP1}
\noindent  This section is devoted to the proof of the existence of
nontrivial normalized solutions of the system \eqref{ODE}. Throughout this section, we assume that
all conditions of \eqref{assumptions} hold and that $0< p <4.$

\subsection{The variational problem}
We study the following problem: for given
$(r,l,m)\in \mathbb{R}_{+}^2\times \mathbb{R},$
find a function $(\phi_1,\phi_2,w)\in S_r\times S_l \times K_m$ such that $E(\phi_1,\phi_2,w)=\Theta(r,l,m),$
where $\Theta(r,l,m)$ is defined by
\begin{equation}\label{deftheta}
\Theta(r,l,m)=\inf\left\{ E(f_1,f_2,g):(f_1,f_2,g) \in S_r\times S_l \times K_m \right\}.
\end{equation}
Before we proceed, let us fix some notations that will be used in the sequel:
\begin{equation*}
\begin{aligned}
& \Sigma_{r,l,m}=S_{r\times l}\times K_m,\ \textrm{where}\ S_{r\times l}=S_r\times S_l, \\
& E_1(f)=E(f,0,0),\ E_2(g)=E(0,g,0),\ E_3(h)=E(0,0,h),\\
& E_{12}(f,g)=E(f,g,0),\ E_{23}(g,h)=E(0,g,h),\ E_{13}(f,h)=E(f,0,h),
\end{aligned}
\end{equation*}
and
\begin{equation*}
F_j(f_j,h)=\alpha_j\int_{-\infty}^{\infty}|f_j|^2h\ dx,\ j=1,2.
\end{equation*}

\begin{lem}\label{Neginf}
The function $\Theta(r,l,m)$
is finite and negative.
\end{lem}
\begin{proof}
Take an element $\Delta=(f_1,f_2,g)$ of $\Sigma_{r,l,m}.$
Using the Gagliardo-Nirenberg inequality and the Cauchy-Schwarz inequality,
it follows that there exists a constant $C=C(q_1, \epsilon)$ such that
\begin{equation}\label{estimate1}
\begin{aligned}
& \|f_1\|_{L^{q_1+2}}^{q_1+2} \leq C\|\partial_xf_{1}\|_{L^2}^{q_1/2}\|f_1\|_{L^2}^{(q_1/2)+2}\\
&\ \ \ \ \leq \epsilon \|\partial_xf_{1}\|_{L^2}^{2}+C\|f_1\|_{L^2}^{2+4q_1/(4-q_1)}=\epsilon \|\partial_xf_{1}\|_{L^2}^{2}+Cr^{1+2q_1/(4-q_1)},
\end{aligned}
\end{equation}
where $\epsilon>0$ can be chosen arbitrarily small.
Using H\"{o}lder's inequality, we obtain that
\begin{equation}\label{estimate2}
\int_{-\infty}^{\infty}|f_1|^2|g|\ dx \leq C \|f_{1}\|_{L^4}^{2}\|g\|_{L^2}\leq \epsilon \|\partial_xf_{1}\|_{L^2}^{2} +Cr(1+t).
\end{equation}
Similar estimates hold for $\|f_2\|_{L^{q_2+2}}^{q_2+2},$ $\|g\|_{L^{p+2}}^{p+2},$ and $\int_{-\infty}^{\infty}|f_2|^2|g|\ dx.$
These estimates will be used repeatedly throughout the rest of the paper.
With the aid of the estimates \eqref{estimate1} and \eqref{estimate2}, one can infer that
\begin{equation*}
E(\Delta)\geq (1-\epsilon)\sum_{j=1}^{2}\|f_j\|_{H^1}^{2}+(1-\epsilon)\|g\|_{H^1}^2-C_{\epsilon,r,l,m}-(r+l+m).
\end{equation*}
Taking $\epsilon<1,$ we now obtain
\begin{equation*}
E(\Delta)\geq -C_{\epsilon,r,l,m}-(r+l+m) > -\infty.
\end{equation*}
To prove $\Theta(r,l,m)<0,$ take any $\Delta=(f_1,f_2,g)\in \Sigma_{r,l,m}$ such
that $f_1(x)>0, f_2(x)>0,$ and $g(x)>0$ for all $x\in \mathbb{R}.$ For arbitrary $\theta>0,$ define the scaling functions
\begin{equation*}
f_{j\theta}(\cdot)=\theta^{1/2}f_j(\cdot \theta), \ j=1,2,\ \textrm{and}\ g_{\theta}(\cdot)=\theta^{1/2}g(\cdot \theta).
\end{equation*}
It is easy to see that $\Delta_\theta=(f_{1\theta},f_{2\theta},g_\theta)$ belongs to $\Sigma_{r,l,m}$ as well and hence, one has
\begin{equation*}
E(\Delta_\theta)\leq \theta^2 \int_{-\infty}^{\infty}\left(|\partial_xf_{1}|^2+|\partial_xf_{2}|^2+|\partial_xg|^2 \right)\ dx - \theta^{1/2}\sum_{j=1}^{2}\alpha_j|f_{j}|^2g\ dx.
\end{equation*}
Upon taking $\theta$ small enough, it is obvious that $E(\Delta_\theta)<0$ and hence,
the infimum $\Theta(r,l,m)$ defined in \eqref{deftheta} is negative.
\end{proof}
\noindent In what follows we call a sequence $\{(f_{1n},f_{2n},g_n)\}$ of functions in $\h$ an $(r,l,m)$-admissible
if the following conditions hold:
\begin{equation*}
\lim_{n\to \infty}Q(f_{1n})=r,\ \ \lim_{n\to \infty}Q(f_{2n})=l,\ \textrm{and}\ \lim_{n\to \infty}Q(g_n)=m.
\end{equation*}
We will say that an $(r,l,m)$-admissible sequence $\{(f_{1n},f_{2n},g_n)\}$ of functions in $\h$ is
a minimizing sequence for $\Theta(r,l,m)$ if it satisfies the condition
\begin{equation}\label{MScondition}
\lim_{n\to \infty}E(f_{1n},f_{2n},g_n)=\Theta(r,l,m).
\end{equation}
Using the estimates obtained in Lemma~\ref{Neginf}, it is easy to prove that
such a sequence is bounded.
The common element in the proof of the relative compactness of minimizing sequence
via concentration compactness argument is to show the
strict subadditivity condition of the problem.
In the present situation, however, this is considerably difficult by the fact that
the function $\Theta(r,l,m)$ consists of three independent parameters.
To overcome this difficulty we utilize the properties of symmetric rearrangement of functions
to carry a careful analysis of minimizing sequences.
In the next few lemmas we will be devoted to proving the strict sub-additivity of $\Theta(r,l,m).$

\smallskip

The first lemma establishes some special properties of minimizing sequences.
\begin{lem}\label{Seqprop}
Suppose $\{(f_{1n},f_{2n},g_n)\}$ be an $(r,l,m)$-admissible sequence satisfying the condition \eqref{MScondition}.
Then the following properties hold:

\smallskip

(i) if $s>0$ and $r,l\geq 0,$ then
there exists a pair of numbers $(\delta_1, N_1)\in (0,\infty)\times \mathbb{N}$ such
that $\|\partial_xg_{n}\|_{L^{2}}\geq \delta_1$ holds for all $n\geq N_1.$

\smallskip

(ii) if $r>0$ and $l, m\geq 0,$ then there
exists a pair $(\delta_2, N_2)\in (0,\infty)\times \mathbb{N}$ such that $\|\partial_xf_{1n}\|_{L^{2}}\geq \delta_2$ for all $n\geq N_2.$
A similar assertion holds for $\|\partial_xf_{2n}\|_{L^2}$ when $r$ and $l$ are switched.
\end{lem}
We prove each part separately.

\smallskip

\noindent \textbf{Proof of Lemma~\ref{Seqprop}~(i).} Suppose, for the sake of contradiction, that
the conclusion of part (i) is
not true. Then, passing to an appropriate subsequence if necessary,
we may assume that $\underset{n \to\infty}{\lim}\|\partial_xg_{n}\|_{L^2}=0.$
Using the estimates obtained in Lemma~\ref{Neginf}, one sees that
\begin{equation*}
\int_{-\infty}^{\infty}|g_n|^{p+2}\ dx \to 0\ \ \textrm{and}\
 \int_{-\infty}^{\infty}|f_{jn}|^2|g_n|\ dx\to 0,\ j=1,2,
\end{equation*}
as $n\to \infty.$ In consequence, we come to the identity
\begin{equation}\label{gpos3}
\Theta(r,l,m)=\lim_{n \to \infty}E(f_{1n},f_{2n},g_n)=\lim_{n\to \infty}\sum_{j=1}^2E_j(f_{jn}).
\end{equation}
Now, take any $ w\in S_m$ such that $w\geq 0.$ For an arbitrary $\theta>0,$
we set $w_\theta(\cdot)=\theta^{1/2}w(\cdot \theta).$ Then the scaling function $w_\theta$ also
belongs to $S_m$ and hence, for all $n,$ there obtains
\begin{equation}\label{partiInq}
\Theta(r,l,m)\leq E(f_{1n},f_{2n},w_\theta).
\end{equation}
On the other hand, for small enough $\theta,$ it is obvious that
\begin{equation}\label{gpos4}
A_w:=\theta^2\int_{-\infty}^{\infty}|\partial_xw|^2\ dx-\tau \theta^{p/2}\int_{-\infty}^{\infty} w^{p+2}\ dx <0,
\end{equation}
Using this notation, one obtains from the inequality \eqref{partiInq} that
\begin{equation*}
\begin{aligned}
\Theta(r,l,m) \leq \sum_{j=1}^{2}\left( E_j(f_{jn})-\theta^{1/2}F_j(f_{jn},w)\right)+ A_w
 \leq \sum_{j=1}^{2}E_j(f_{jn}) + A_w.
\end{aligned}
\end{equation*}
Upon passing limit as $n\to \infty$ to the last inequality, we obtain that
\begin{equation*}
\Theta(r,l,m)\leq \lim_{n\to \infty}\sum_{j=1}^{2}E_j(f_{jn}) + A_w,
\end{equation*}
but this last inequality contradicts \eqref{gpos3} and \eqref{gpos4}, and hence, the result (i) follows. \hfill $\Box$

\smallskip

\noindent To prove part (ii), we shall make use of the following result of \cite{[AB11]}
concerning the existence of solutions of a certain problem closely related to \eqref{deftheta}.

\begin{lem}\label{nlskdvAB}
For $(\phi, w)\in H_{\mathbb{C}}^1\times H_{\mathbb{R}}^1,$
define $E_{23}(\phi,w)=E(0,\phi,w).$ Then, for every $l>0$ and $m>0,$ there exists a solution $(\phi_0,w_m)$ to the problem
\begin{equation*}
\nu(l,m)=\inf\{E_{23}(\phi,w):\phi\in S_l\ \textrm{and}\ w\in K_m\}.
\end{equation*}
Furthermore, $w_m(x)>0$ for all $x\in \mathbb{R}$ and there exists a positive $\mathbb{R}$-valued
function $\phi_l$ such that $\phi_0=\phi_l$ up to a phase factor. In particular, $E_{23}(\phi_l,w_m)=\nu(l,m).$
\end{lem}

The next two lemmas are well-known uniqueness results  (for details and further discussion, we refer readers to \cite{[Ca]}).
\begin{lem}\label{KdVSOe}
Suppose $W_{p}\in H_{\mathbb{R}}^1$ is a non-zero solution of
\begin{equation*}
-Q^{\prime \prime}+\lambda_3 Q = \frac{\beta}{p+1}Q^{p+1},
\end{equation*}
where $\lambda_3\in \mathbb{R}.$ Then $\lambda_3>0$ and $W_p(x)=w(x+x_0),$ where $x_0\in \mathbb{R}$ and $w_p(x)$
has the following explicit expression
\begin{equation}
w_p(x)=\left(\frac{(p+1)(p+2)\lambda_3}{2\beta} \right)^{1/p}\textrm{sech}^{2/p}\left(\frac{\sqrt{\lambda_3}~p x}{2}\right).
\end{equation}
\end{lem}
\noindent The second lemma concerns about the uniqueness of solutions of the
equations,
\begin{equation}\label{NLSexplilem}
\left\{
\begin{aligned}
-&Q_j^{\prime \prime}+\lambda_j Q_j=(q_j+2)\tau_j~|Q_j|^{q_j+1},\\
& Q_j \in H_{\mathbb{C}}^1\setminus \{0\},\ \lambda_j\in \mathbb{R},\ j=1,2.
\end{aligned}
\right.
\end{equation}

\begin{lem}\label{NLSSOe}
There is, up to translation and phase shift, a unique solution of the equation \eqref{NLSexplilem}, i.e.,
the set of all solutions of \eqref{NLSexplilem} is of the form
\begin{equation*}
\mathcal{G}_j=\left\{e^{i\theta_j}\psi_{q_j}(\cdot +x_0): \theta_j, x_0\in \mathbb{R}\right\},
\end{equation*}
where $\psi_{q_j}(x)$ is explicitly given by
\begin{equation}
\psi_{q_j}(x)=\left(\frac{\lambda_j}{2\tau_j} \right)^{1/q_j}\textrm{sech}^{2/q_j}\left(\frac{\sqrt{\lambda_j}~q_j x}{2}\right).
\end{equation}
\end{lem}

\smallskip

We are now able to prove part (ii) of Lemma~\ref{Seqprop}.

\smallskip

\noindent \textbf{Proof of Lemma~\ref{Seqprop}~(ii).} The proof is again carried out by contradiction.
As before, by extracting a subsequence if necessary, one assumes that $\underset{n\to \infty}{\lim}\|\partial_xf_{1n}\|_{L^2}=0.$
Then, using the estimates \eqref{estimate1} and \eqref{estimate2} yet again, one easily verifies that
\begin{equation*}
\lim_{n\to \infty}\int_{-\infty}^{\infty}|f_{1n}|^2|g_n|\ dx =0= \lim_{n\to \infty}\int_{-\infty}^{\infty} |f_{1n}|^{q_1+2}\ dx,
\end{equation*}
and as a result, we come to the identity
\begin{equation}\label{Inqcase2}
\Theta(r,l,m)=\lim_{n\to \infty}\left[ E_2(f_{2n})+E_3(g_n)-F_{2}(f_{2n},g_n)\right].
\end{equation}
The following four cases are possible: (i) $l>0$ and $m>0\ ;$ (ii) $l>0$
and $m=0\ ;$ (iii) $l=0$ and $m>0\ ;$ and (iv) $l=0$ and $m=0 .$
If $(l,m)\in \mathbb{R}_{+}^2,$ let the functions $\phi_l$ and $w_m$ be as defined in Lemma~\ref{nlskdvAB}. Then
$\Theta(r,l,m) \geq E_{23}(\phi_l,w_m)$. To arrive at a contradiction,
it is claimed that there exists $R_r\in S_r$ such that
\begin{equation}\label{Neg1str}
\|\partial_xR_{r}\|_{L^2}^2 -F_1(R_r, w_m) < 0.
\end{equation}
To prove \eqref{Neg1str}, take any $\rho\in C_c^{\infty}$ such that $\rho\geq 0,$ $\rho(0)=1,$ and $\|\rho\|_{L^2}^2=r.$
For an arbitrary $\theta > 0,$ we set $\rho_{\theta}(x)=\theta^{1/2}\rho(\theta x).$ Then $\rho_\theta \in S_r.$
Since the function $w_m$ is integrable over $\mathbb{R},$ an application of Lebesgue dominated convergence theorem yields
\begin{equation*}
\|\partial_x\rho_{\theta } \|_{L^2}^2-F_1(\rho_\theta, w_m) \leq \theta^2\int_{-\infty}^{\infty}|\partial_x\rho|^2\ dx-\theta \int_{-\infty}^{\infty}w_m(x)\ dx.
\end{equation*}
Since $\int_{-\infty}^{\infty}w_m(x)\ dx>0,$ it follows that $\|\partial_x\rho_{\theta} \|_{L^2}^2-F_1(\rho_\theta, w_m)<0$
for sufficiently small $\theta.$ Thus
$R_r=\rho_{\theta}$ satisfies \eqref{Neg1str} whenever $\theta$ is sufficiently small.
With \eqref{Neg1str} in hand, we now obtain
\begin{equation*}
\begin{aligned}
\Theta(r,l,m)& \leq E(R_r,\phi_l,w_m)\\
& =E_{23}(\phi_l,w_m)+\left( \|\partial_xR_{r}\|_{L^2}^2 -F_1(R_r, w_t)\right)-\tau_1 \int_{-\infty}^{\infty} |R_r|^{q_1+2}\ dx\\
& \leq E_{23}(\phi_l,w_m)+ \left( \|\partial_xR_{r}\|_{L^2}^2 -F_1(R_r, w_m)\right) < E_{23}(\phi_l,w_m),
\end{aligned}
\end{equation*}
which is a contradiction. Next, consider the case that $l>0$ and $m=0.$
The uniqueness result (Lemma~\ref{NLSSOe}) implies that for any $l>0,$ any solution
of the problem
$\inf\{E_2(g):g\in S_l\}$
is of the form $\psi_l=e^{i\theta_0}T_{x_0}\psi_{q_2},$
where $\theta_0, x_0\in \mathbb{R}$ and $\psi_{q_2}$ is as defined in Lemma~\ref{NLSSOe}.
Then, from \eqref{Inqcase2}, we have that
$\Theta(r,l,m) \geq E_2(\psi_l).$
To arrive at a contradiction,
take any $(f,w)\in S_r\times K_m$ such that $f(x)>0$ and $w(x)>0$ for $x\in \mathbb{R}.$ For an arbitrary $\theta>0,$ define
$f_\theta(x)=\theta^{1/2}f(\theta x)$ and $w_\theta(x)=\theta^{1/2}w(\theta x).$ Then it is obvious
that $(f_\theta,w_\theta)\in S_r\times K_m$ and hence,
\begin{equation*}
\begin{aligned}
E_{13}(f_\theta, w_\theta)& =E_1(f_\theta)+E_3(w_\theta)-F_1(f_\theta, w_\theta) \\
& \leq \theta^2 \int_{-\infty}^{\infty}\left(|\partial_xf|^2+|\partial_xw|^2 \right)\ dx - \theta^{1/2}\int_{-\infty}^{\infty}\alpha_1|f|^2w\ dx,
\end{aligned}
\end{equation*}
from which it is concluded that $E_{13}(f_\theta, w_\theta)<0$ whenever $\theta$ is sufficiently small.
In consequence of the preceding inequality, one has that
\begin{equation*}
\begin{aligned}
\Theta(r,l,m)& \leq E(f_\theta,\psi_l,w_\theta)=E_2(\psi_l)+E_{13}(f_\theta, w_\theta)-F_2(\psi_l,w_\theta) \\
& \leq E_2(\psi_l)+E_{13}(f_\theta, w_\theta)<E_2(\psi_l),
\end{aligned}
\end{equation*}
a contradiction. The case that $l=0$ and $m>0$ is similar. Finally, suppose that $l=0$ and $m=0.$
Then, from \eqref{Inqcase2}, one has $\Theta(r,l,m)\geq 0.$ On the other hand,
\begin{equation*}
\Theta(r,l,m)=\inf\left\{E_1(f): f\in S_r\right\},
\end{equation*}
and we can make $E_1(f)<0$ by taking the scaling function $f_\theta(x)=\theta^{1/2}f(\theta x)$ defined as before.
This in turn implies that $\Theta(r,l,m)<0,$ a contradiction.
This completes the proof in all cases. The conclusion for $\partial_xf_{2n}$ can be proved by using an analogous argument. \hfill $\Box$

\medskip

\noindent Another important ingredient for the proof of strict sub-additivity is the following lemma, which concerns the
existence of special minimizing sequences for $\Theta(r,l,m).$

\begin{lem}\label{specialmin}
There exist a minimizing sequence $\{(f_{1n},f_{2n},g_n)\}$ for $\Theta(r,l,m)$ such that
for each $n,$ the functions $f_{1n}, f_{2n},$ and $g_n$ are $\mathbb{R}$-valued, non-negative, $C_c^{\infty},$ even,
non-increasing on the set $[0,\infty),$ and satisfy the condition
\begin{equation*}
\|f_{1n}\|_{L^2}^{2}=r,\ \|f_{2n}\|_{L^2}^{2}=l,\ \textrm{and}\ \|g_{n}\|_{L^2}^{2}=m.
\end{equation*}
\end{lem}
\begin{proof}
Start with a given minimizing sequence $(p_{1n},p_{2n},q_{n})$ for $\Theta(r,l,m).$
Without loss of generality we may assume that $(r,l,m)\in \mathbb{R}_{+}^3.$
First approximate $(p_{1n},p_{2n},q_{n})$
by compactly supported functions $(c_{1n},c_{2n},d_{n}).$
For a non-negative measurable function $f,$ let $f^\ast$ denotes its symmetric decreasing
rearrangement (for details, see Chapter~3 of \cite{[LL]}).
Using rearrangement inequalities (cf. Chapter~7 of \cite{[LL]}),
we have that
\begin{equation*}
E(|f_1|^\ast,|f_2|^\ast,|g|^\ast)\leq E(f_1,f_2,g)
\end{equation*}
for any $(f_1,f_2,g)\in \h.$
Hence, one can assume without loss of generality that $c_{1n}=|c_{1n}|^\ast,$ $c_{2n}=|c_{2n}|^\ast,$
and $d_{n}=|d_{n}|^\ast$ hold.
Now let $\psi \in C^\infty_{c}$ be any non-negative, even, and decreasing function on the set $[0,\infty),$ which also satisfies the condition
$\int_{-\infty}^{\infty}\psi(x)\ dx=1.$ For any arbitrary $\epsilon>0,$
consider $\psi_\epsilon(\cdot)=(1/\epsilon)\psi(\cdot/\epsilon),$ and set
\begin{equation*}
f_{1n}=\frac {r^{1/2} \left(c_{1n}\star \psi_{\epsilon_n}\right)} {\|c_{1n}\star \psi_{\epsilon_n}\|_{L^2}}, \
\ \ f_{2n}=\frac {s^{1/2} \left(c_{2n}\star \psi_{\epsilon_n}\right)} {\|c_{2n}\star \psi_{\epsilon_n}\|_{L^2}},\
g_{n} = \frac {t^{1/2} \left(d_{n}\star \psi_{\epsilon_n}\right)} {\|d_{n}\star \psi_{\epsilon_n}\|_{L^2}},
\end{equation*}
with $\epsilon_n$ chosen approximately small whenever $n$ is large.
Since $\psi_{\epsilon_n}$ is a mollifier,
it follows that the sequence
$\{(f_{1n}, f_{2n}, g_{n})\}$ is the desired minimizing
sequence.
\end{proof}

\noindent With the properties of
minimizing sequences given in Lemmas~\ref{Seqprop} and \ref{specialmin} in hand, we are now able to
prove the strict subadditivity inequality for the function $\Theta(r,l,m).$ This will be an essential ingredient later
in ruling out the case of dichotomy.

\begin{lem} \label{subadd}
The function
$\Theta(r,l,m)$ enjoys
the following strict subadditivity property
\begin{equation}
\label{SUBA}
\Theta(r_{1}+r_{2},l_{1}+l_{2},m_{1}+m_{2})<\sum_{i=1}^{2}\Theta(r_{i},l_{i},m_{i}),\ r_i, l_i, m_i \geq 0.
\end{equation}
\end{lem}
\begin{proof}
We may assume that $r_1+r_2>0,$\ $l_1+l_2>0,$\ $m_1+m_2>0,$
$r_1+l_1+m_1>0,$ and $r_2+l_2+m_2>0;$ otherwise \eqref{SUBA} reduces to the strict
subadditivity inequality of the function with fewer parameters.
For $i=1,2,$ consider the special minimizing sequences
$\{(f_{1n}^{(i)}, f_{2n}^{(i)}, g_{n}^{(i)})\}$ for $\Theta(r_i,l_i,m_i),$ as constructed in Lemma~\ref{specialmin}.
For each $n$, select the numbers $x_n$ such that
such that
\begin{equation*}
\textrm{supp}\ f_{jn}^{(1)}\cap \textrm{supp}\ T_{x_n}f_{jn}^{(2)}=\emptyset\ \ \textrm{and}\ \textrm{supp}\ g_{n}^{(1)}\cap \textrm{supp}\ T_{x_n}g_{n}^{(2)}=\emptyset,
\end{equation*}
and define the sequence $\{(f_{1n},f_{2n},g_n)\}$ of functions by setting
\begin{equation}\label{deffgSUBA}
(f_{1n},f_{2n},g_n)=\left( \left(f_{1n}^{(1)}+T_{x_n}f_{1n}^{(2)}\right)^\ast, \left(f_{2n}^{(1)}+T_{x_n}f_{2n}^{(2)}\right)^\ast, \left(g_{n}^{(1)}+T_{x_n}g_{n}^{(2)} \right)^\ast\right).
\end{equation}
By the definition of the infimum $\Theta(r_1+r_2,l_1+l_2,m_1 + m_2),$ it is clear that
\begin{equation}
\Theta(r_1+r_2,l_1+l_2,m_1 + m_2) \le E(f_{1n},f_{2n},g_n).
\label{IleE}
\end{equation}

\noindent A lemma about symmetric rearrangement (Lemma~2.10 of \cite{[AB11]}) now comes to our aid. The lemma
states that if $f,g:\mathbb{R}\to [0,\infty)$ are non-increasing, even, $C_c^\infty$ functions;
the real numbers
$x_1, x_2$ be such that the translated functions $T_{x_1}f$ and
$T_{x_2}g$ have disjoint supports; and
$S=T_{x_1}f+T_{x_2}g,$
then the first derivative $(S^\ast)'$ of
$S^\ast$ (in the distributional sense) is in $L^2$ and one has the estimate
\begin{equation}
\|(S^\ast)'\|_{L^2}^2 \le \|S'\|_{L^2}^2 - \frac34 \min\{\|f'\|_{L^2}^2,\|g'\|_{L^2}^2\}.
\label{garineq}
\end{equation}

\smallskip

\noindent Applying the estimate \eqref{garineq} to each component of the sequence \eqref{deffgSUBA} and using the well-known
fact that $\|u\|_{L^p}=\|u^\ast\|_{L^p}, 1\leq p \leq \infty,$ it immediately follows that
\begin{equation}
\begin{aligned}
\|\partial_xg_{n}\|_{L^2}^2+ & \sum_{j=1}^2\|\partial_xf_{jn}\|_{L^2}^2
 \le \|\partial_xg_{n}^{(1)}\|_{L^2}^2+ \|\partial_x(T_{x_n}g_{n}^{(2)})\|_{L^2}^2 \\
& + \sum_{j=1}^{2}\left( \|\partial_xf_{jn}^{(1)}\|_{L^2}^2+\|\partial_x(T_{x_n}f_{jn}^{(2)})\|_{L^2}^2 \right)\\
& -\frac{3}{4}\left(\min\left\{\|\partial_xg_{n}^{(1)}\|_{L^2}^2,\|\partial_xg_{n}^{(2)}\|_{L^2}^2\right\}
+ \sum_{j=1}^{2} \min\left\{\|\partial_xf_{jn}^{(1)}\|_{L^2}^2,\|\partial_xf_{jn}^{(2)}\|_{L^2}^2\right\}\right).
\end{aligned}
\label{kinenstrictdec}
\end{equation}
Now estimating the right side of \eqref{IleE} using \eqref{kinenstrictdec} and the
rearrangement inequality (Chapter 3 of \cite{[LL]}), and passing to the limit as $n\to \infty$ in the resultant inequality,
there obtains
\begin{equation}
\Theta(r_1+r_2, l_1+l_2, m_1+m_2) \le \sum_{i=1}^{2}\Theta(r_i,l_i,m_i) - \liminf_{n \to \infty} J_n,
\label{subaddKn}
\end{equation}
where the quantity $J_n$ denotes the last term in \eqref{kinenstrictdec} involving minimums.
We now prove the strict inequality \eqref{SUBA}. As noted in \cite{[SanDCDS]}, it is sufficient to consider the following cases:
\begin{enumerate}
\item[(i)] $r_1,r_2>0$ and $l_1,l_2,m_1,m_2\geq 0 ;$
\item[(ii)] $r_1=0, r_2>0, l_2>0,\ \text{and} \ m_1=0 ;$
\item[(iii)] $r_1=0, r_2>0, l_2>0, \ \text{and} \ m_1>0;$
\item[(iv)] $r_1=0, r_2>0, l_2=0, \ \text{and} \ m_1=0 ;$ and
\item[(v)] $r_1=0, r_2>0, l_2=0, \ \text{and} \ m_1>0.$
\end{enumerate}
All other cases can be reduced to one of these cases by switching the roles of the parameters.
Consider first the case that $r_1, r_2 > 0.$ Using
Lemma \ref{Seqprop}(ii), there exist a pair of positive numbers $\{\delta_1,\delta_2\}$
such that for all large enough $n$, we have
$\|\partial_xf_{1n}^{(1)}\|_{L^2} \ge \delta_1$ \ \text{and}\ $\|\partial_xf_{1n}^{(2)}\|_{L^2} \ge
\delta_2.$
Let $\delta =
\min(\delta_1, \delta_2)> 0$. Then it follows that $J_n \ge
3\delta/4$ for all large enough $n,$ and in view of \eqref{subaddKn},
we conclude that
\begin{equation*}
\begin{aligned}
\Theta(r_1+r_2, l_1+l_2, m_1+m_2) &\le \Theta(r_1,l_1,m_1)+\Theta(r_2,l_2,m_2) - 3\delta/4 \\
&< \Theta(r_1,l_1,m_1)+\Theta(r_2,l_2,m_2). \label{subaddwdelta}
\end{aligned}
\end{equation*}
Next suppose the case that $r_1=0, r_2>0, l_2>0,$ and $m_1=0.$
Since $r_1+l_1+m_1>0,$ so $l_1>0$ as well. Then another application of Lemma~\ref{Seqprop}(ii) guarantees the existence of
numbers $\delta_3, \delta_4 >0$
such that for all large enough $n$,
$\|\partial_xf_{2n}^{(1)}\|_{L^2} \ge \delta_3$ \ \text{and}\ $\|\partial_xf_{2n}^{(2)}\|_{L^2} \ge
\delta_4.$
As before, set $\delta =
\min(\delta_3, \delta_4)> 0$. Then it is obvious that $J_n \ge
3\delta/4$ for all large enough $n,$ and \eqref{SUBA} follows from \eqref{subaddKn}.
This completes the proof in case (ii).

\smallskip

\noindent We now turn to the case (iii), i.e., when $r_1=0, r_2>0, l_2>0, \ \text{and} \ m_1>0.$
If $l_1>0$ or $m_2>0,$ then the proofs follow same lines as that in the case (ii) above. Thus, we assume
$l_1=0 \ \text{and} \ m_2=0,$ and prove that
\begin{equation}\label{Suba3}
\Theta(r_2,l_2,m_1) < \Theta(0,0,m_1)+\Theta(r_2,l_2,0).
\end{equation}
Take the function $w_p$ as defined in Lemma~\ref{KdVSOe} with $\lambda_3>0$ so chosen
such that $w_p \in K_{m_1}.$ Then the function
$w_{m_1}=w_p$ satisfies the identity
\begin{equation*}
E_3(w_{p})=\inf\{E_3(h): h\in K_{m_1} \}
\end{equation*}
(see, for example \cite{[Ca]}).
Similarly, let $\psi_{q_1}, \psi_{q_2}$ be the functions as defined in
Lemma~\ref{NLSSOe} with $\lambda_1$ and $\lambda_2$ so chosen such that $(\psi_{q_1},\psi_{q_2})\in S_{{r_2}\times {l_2}}.$
Then the functions $\phi_{r_2}=\psi_{q_1}$ and $\phi_{l_2}=\psi_{q_2}$ satisfy the identities
\begin{equation*}
\begin{aligned}
E_1(\phi_{r_2})=\inf\{E_1(f): f\in S_{r_2} \}\ \textrm{and}\
E_2(\phi_{s_2})=\inf\{E_2(g): g\in S_{l_2} \}.
\end{aligned}
\end{equation*}
Now the function $(\phi_{r_2},\phi_{l_2},w_{m_1})$ belongs to $S_{{r_2}\times {l_2}}\times K_{m_1}$ and we come to the inequality
\begin{equation}\label{subcase3}
\begin{aligned}
& \Theta(r_2,l_2,m_1)\leq E(\phi_{r_2},\phi_{l_2},w_{m_1})=\int_{-\infty}^{\infty} \left(|\partial_x(w_{m_1})|^{2}-\tau w_{m_1}^{p+2}\right)dx \\
& + \int_{-\infty}^{\infty} \left( |\partial_x(\phi_{r_2})|^{2}+|\partial_x(\phi_{l_2})|^{2}-\tau_1|\phi_{r_2}|^{q_1+2}-\tau_2|\phi_{l_2}|^{q_2+2}
 \right)dx \\
& - \alpha_1 \int_{-\infty}^{\infty}|\phi_{r_2}|^{2} w_{m_1} \ dx
- \alpha_2 \int_{-\infty}^{\infty}|\phi_{l_2}|^{2} w_{m_1} \ dx,
\end{aligned}
\end{equation}
It is obvious that
$\int_{-\infty}^{\infty}|\phi_{r_2}|^{2} w_{m_1}\ dx >0$ \ \textrm{and} \ $\int_{-\infty}^{\infty}|\phi_{l_2}|^{2} w_{m_1}\ dx >0.$
Then, the strict inequality \eqref{Suba3} follows from \eqref{subcase3}.
In the case (iv), one has to prove that
\begin{equation}\label{Suba2}
\Theta(r_2,l_1,m_2)<\Theta(0,l_1,0) + \Theta(r_2,0,m_2),
\end{equation}
which can be done using an analogous argument as in the proof of \eqref{Suba3}.
It only remains to prove \eqref{SUBA} in case (v).
Consider the case (v), i.e., $r_1=0, r_2>0, l_2=0, \ \text{and} \ m_1>0.$
If $t_2>0,$ then \eqref{SUBA} follows by the use of part (i) of
Lemma~\ref{Seqprop} in the inequality \eqref{subaddKn}. Thus,
one can assume $t_2=0$ and prove that
\begin{equation}\label{Suba4}
\Theta(r_2,l_1,m_1)<\Theta(0,l_1,m_1) + \Theta(r_2,0,0).
\end{equation}
The proof of the inequality \eqref{Suba4} follows along the same lines as that of \eqref{Suba3} as well.
This completes the proof of \eqref{SUBA} in all cases.
\end{proof}

We now proceed to prove the existence result for normalized solutions.

\subsection{Existence result for (2+1)-component NLS-gKdV}

To any minimizing sequence $\{(f_{1n},f_{2n},g_n)\}$ is associated, up to taking a subsequence,
 a number $\mu$ given by
\begin{equation}\label{defmu}
\mu=\lim_{\zeta \to \infty} \lim_{n\to \infty}\sup_{y\in \mathbb{R}}\int_{y-\zeta}^{y+\zeta}\rho _{n}(x)\ dx,
\end{equation}
where the function $\rho _{n}(x)$ is defined by
\begin{equation*}
 \rho _{n}(x):=|f_{1n}(x)|^2 + |f_{2n}(x)|^2+g_{n}^2(x).
 \end{equation*}
Then the number $\mu$ satisfies $0\leq \mu \leq r+l+m.$
We will examine separately the three (mutually exclusive)
cases, $\mu = r+l+m$ (tightness),\ $0<\mu<r+l+m$ (dichotomy), and $\mu=0$ (vanishing).
Once we prove the tightness $\mu = r+l+m,$ then one can follow the same lines as in
the proof of the fundamental Lemma~I.1(i) of \cite{[L1]} to prove that
the translated sequence $\{(T_{y_n}f_{1n}, T_{y_n}f_{2n},T_{y_n}g_{n})\}$
has a subsequence which converges in $\h$ norm to a function in $\mathcal{O}_{r,l,m}.$ The proof differs
only in minor details and will not be repeated here.
Thus, in order to
prove Theorem~\ref{SOexistence}, it
suffices to rule out dichotomy and vanishing cases.

\begin{lem} \label{RSUBA}
Suppose $(r,l,m)\in \mathbb{R}_{+}^3$ and let $\{(f_{1n},f_{2n},g_n)\}$ be an $(r,l,m)$-admissible sequence
satisfying \eqref{MScondition}.
Let $\mu$ be as defined in $\eqref{defmu}.$ Then
there exists $(r_1,l_1,m_1)\in [0,r]\times [0,l] \times [0,m]$ satisfying
$\mu = r_1+l_1+m_1$
and
\begin{equation} \label{REV}
\Theta(r_1, l_{1},m_{1})+\Theta(r-r_{1}, l-l_{1},m-m_{1})\leq \Theta(r,l,m).
\end{equation}
\end{lem}
\begin{proof}
The proof follows along the same lines as that
of Theorem~3.4 of \cite{[SanDCDS]}, which is a generalization of Theorem~3.10 of \cite{[AA]}. We omit the details.
\end{proof}

We are now able to prove the existence theorem.

\smallskip

\noindent \textbf{Proof of Theorem~\ref{SOexistence} (existence of prescribed $L^2$-norm solutions).} Suppose $(r,l,m)\in \mathbb{R}_{+}^3$ and let $\{(f_{1n},f_{2n},g_n)\}$ be any minimizing sequence
for $\Theta(r,l,m).$ As noted before, the existence of minimizers follows if we show that
$\mu =r+l+m,$
where $\mu$ is as defined in $\eqref{defmu}.$ Hence, to complete the proof of Theorem~\ref{SOexistence}, we only have to show that
\begin{itemize}
\item[(i)] $\mu\neq 0,$\ and
\item[(ii)] $\mu \not\in (0,r+l+m).$
\end{itemize}
In view of Lemmas~\ref{subadd} and \ref{RSUBA}, (ii) is clear.
So we only have to prove that $\mu\neq 0.$
A standard fact in the application of concentration compactness method (Lemma~I.1 of \cite{[L1]}) states
that if $\{u_n\}$ is bounded in $L^\alpha,$ $\{u_n^{\prime}\}$ is bounded in $L^p,$ and for some $\omega>0,$
\begin{equation}\label{lem4van2}
\lim_{n\to \infty}~\sup_{y\in \mathbb{R}}\int_{y-\omega}^{y+\omega}|u_n|^\alpha \ dx =0,
\end{equation}
then for every $q>\alpha,$
\begin{equation}\label{lem4van}
\lim_{n\to \infty}\int_{-\infty}^{\infty}|u_n|^q\ dx =0.
\end{equation}
If $\mu=0,$ then \eqref{lem4van2} holds for $u_n=|f_{1n}|,$ $u_n=|f_{2n}|,$ and $u_n=g_{n},$ and for every $\alpha >2,$
$f_{1n}, f_{2n},$ and $g_n$ all converge to $0$ in $L^\alpha.$
But then, since
\begin{equation*}
|F_j(f_{jn},g_n)|\leq \|f_{jn}\|_{L^4}^{1/2}\|g_n\|_{L^2},\ j=1,2,
\end{equation*}
and $\|g_n\|_{L^2}$ stays bounded, we have that $F_j(f_{jn},g_n)\to 0$ as $n\to \infty.$ As a result,
\begin{equation*}
\begin{aligned}
\Theta(r,l,m)& =\lim_{n\to \infty}E(f_{1n},f_{2n},g_n)\\
& \geq \liminf_{n\to \infty}\int_{-\infty}^{\infty}\left(|\partial_xf_{1n}|^2+|\partial_xf_{2n}|^2+|\partial_xg_{n}|^2 \right)\ dx \geq 0
\end{aligned}
\end{equation*}
which contradicts Lemma~\ref{Neginf} and hence, $\mu\neq 0.$

\smallskip

\noindent The fact that the complex-valued function $\phi_1$ is of
the form $\phi_1(x)=e^{i\theta_1}\tilde{\phi}_{1}(x)$ with $\theta_1\in \mathbb{R}$ and
$\tilde{\phi}$ real-valued nonnegative function can be easily proved by using the first equation of \eqref{SO}.
(A proof of this fact is given in Theorem~2.1 of \cite{[AA]} for $\gamma_1=0$ and
in Theorem~3.7 of \cite{[SanDCDS]} for $\gamma_1>0$; same proof works in the present situation.)
Similarly $\phi_2(x)=e^{i\theta_2}\tilde{\phi}_{2}(x)$ with $\theta_2\in \mathbb{R}$ and $\tilde{\phi_2}\geq 0$ on $\mathbb{R}.$
To continue the proof, we need the following result.

\smallskip

\underline{Claim.} Suppose $\{(f_{1n},f_{2n},g_n)\}\subset \mathcal{H}$ be an $(r,l,m)$-admissible
sequence satisfying the condition \eqref{MScondition}.
If $(r,l,m)\in \mathbb{R}_{+}^3,$ then there exists $\delta_1,\delta_2>0$ such that
for all $n$ large enough, one has that
\begin{equation*}
E_j(f_{jn})-F_j(f_{jn},g_n) \leq -\delta_j,\ j=1,2.
\end{equation*}
To see this, we take $j=1;$ the proof for $j=2$ follows same argument.
We argue by contradiction.
After choosing an appropriate subsequence if necessary,
assume that there exists a minimizing sequence
$\{(f_{1n},f_{2n},g_n)\}$ that satisfies
\begin{equation}
\liminf_{n \to \infty} \left[ E_1(f_{1n})-F_{1}(f_{1n},g_n)\right] \geq 0.
\end{equation}
This in turn implies that
\begin{equation}
\begin{aligned}
 \Theta(r,l,m)& = \lim_{n\to\infty} E(f_{1n},f_{2n},g_n) \\
& \ge \liminf_{n \to \infty}
 \left[E_2(f_{2n})+E_3(g_n)-F_{2}(f_{2n},g_n)\right].
\end{aligned}
 \label{Lemiv1}
\end{equation}
Take the functions $\phi_l$ and $w_m$ as defined in Lemma~\ref{nlskdvAB}. Then, in view of \eqref{Lemiv1}, we have that
$\Theta(r,l,m) \geq E_{23}(\phi_l,w_m).$
On the other hand, as in the proof of part (ii) of Lemma~\ref{Seqprop}, take any $R_r\in S_r$ satisfying
\begin{equation*}
\|\partial_xR_{r}\|_{L^2}^2-F_{1}(R_r,w_m) < 0.
\end{equation*}
Using this inequality, it is deduced that
\begin{equation*}
\begin{aligned}
\Theta(r,l,m)\leq E(R_r,\phi_l,w_m) & \leq E_{23}(\phi_l,w_m)+ \left(\|\partial_xR_{r}\|_{L^2}^2-F_{1}(R_r,w_m) \right)\\
& < E_{23}(\phi_l,w_m),
\end{aligned}
\end{equation*}
a contradiction. This completes the proof the claim.

\smallskip

\noindent Next, multiply the first and second equations of \eqref{ODE} by $\overline{\phi_1}$ and $\overline{\phi_2},$ respectively,
and integrate over the real line. After suitable integrations by parts, it follows immediately from
the above claim that $\sigma_1>0$ and $\sigma_2>0.$
To prove the remaining assertions of Theorem~\ref{SOexistence}, we borrow an argument from \cite{[AB11]}.
Since $\sigma_1>0$ and $\sigma_2>0,$ the first two equations in \eqref{ODE} can be rewritten in the following convolution form
\begin{equation}\label{convrepre}
\phi_j=P_{\sigma_j}\star \left(\gamma_j |\phi_j|^{q_j}\phi_j+\alpha_j\phi_jw \right),\ j=1,2,
\end{equation}
where for any $a>0,$ the kernel $P_a$ is defined via $\widehat{P}_a(k)=(s+k^2)^{-1}.$
Next using the fact that
\begin{equation*}
E(|\phi_1|,|\phi_2|,|w|)=E(|\phi_1|,|\phi_2|,w)=\Theta(r,l,m),
\end{equation*}
one can show that
\begin{equation}\label{posweq}
\int_{-\infty}^{\infty}|\phi_j|^2|w|\ dx = \int_{-\infty}^{\infty}|\phi_j|^2 w\ dx,\ j=1,2,
\end{equation}
(for details, readers may consult \cite{[AB11]}). Then, the identity \eqref{posweq} implies
that $w(x)\geq 0$ at all $x\in \mathbb{R}$ for which $\tilde{\phi_1}(x)\neq 0$ and $\tilde{\phi_2}(x)\neq 0.$
It then follows from the convolution identity \eqref{convrepre} that $\tilde{\phi_1}(x)>0$ and $\tilde{\phi_2}(x)>0$ for all $x\in \mathbb{R}.$
The proof that $w(x)>0$ goes through unchanged
as in the proof of Theorem~1.1~(iv) of \cite{[AB11]} and so will not be repeated here. \hfill $\Box$

\section{Stability Analysis for Solitary Waves}\label{fullvar}
\noindent In this section, consideration is given to the full variational problem \eqref{Vardef}.
To prove the existence of solutions to the problem \eqref{Vardef}, we establish
a relation between the solutions to \eqref{deftheta} and \eqref{Vardef}, following the
arguments of \cite{[AA],[AB11]}. Throughout this section, we assume that all conditions of \eqref{assumptions} hold and that
$1\leq p < 4/3.$
\subsection{The full variational problem}
We begin by showing that every minimizing sequence for $\Lambda(r,l,m)$ is bounded.
By a minimizing sequence for the problem \eqref{Vardef} we mean
a sequence $\{(h_{1n},h_{2n},g_n)\}\subset \h$
satisfying the conditions
\begin{equation*}
\lim_{n\to \infty}Q(h_{1n})=r,\ \ \lim_{n\to \infty}Q(h_{2n})=l,\ \ \lim_{n\to \infty}H(h_{1n},h_{2n},g_n)=m,
\end{equation*}
and
\begin{equation*}
\lim_{n\to \infty}E(h_{1n},h_{2n},g_n)=\Lambda(r,l,m).
\end{equation*}

\begin{lem}\label{boundfull}
If $\{(h_{1n},h_{2n},g_n)\}$ is a
minimizing sequence for \eqref{Vardef}, then there exists a constant $B>0$ such that
\begin{equation*}
\|h_{1n}\|_{H^1}+\|h_{2n}\|_{H^1}+\|g_{n}\|_{H^1}\leq B,\ \ \textrm{for all}\ n.
\end{equation*}
\end{lem}
\begin{proof}
We begin by estimating the sum of the component masses.
Because $Q(h_{jn}), j=1,2,$ stay bounded, it then follows that
\begin{equation}\label{bound1}
\begin{aligned}
J_2\equiv \|g_n\|_{L^2}^2& +\sum_{j=1}^{2}\|h_{jn}\|_{L^2}^2 =\left\vert H(\Delta_n)-2\textrm{Im}\int_{-\infty}^{\infty}\sum_{j=1}^{2}h_{jn}\overline{\partial_xh_{jn}}\ dx \right\vert + \sum_{j=1}^{2}\|h_{jn}\|_{L^2}^2\\
& \leq C\left(1+\sum_{j=1}^{2} \|h_{jn}\|_{L^2}^2 \right)+ \sum_{j=1}^{2}\|h_{jn}\|_{L^2}^2 \leq C\left(1+\|\Delta_n\|_\h \right),
\end{aligned}
\end{equation}
where $\Delta_n=(h_{1n},h_{2n},g_n)$ and $C=C(r,l,m).$
Define the quantity $\mathcal{L}_{p}(bx,cy)=b|x|^{p+2}+c|x|^2y.$ Then it follows directly from \eqref{bound1} that
\begin{equation}\label{mainest}
\begin{aligned}
& \|\Delta_n\|_{\h}^{2} =E(\Delta_n)+\int_{-\infty}^{\infty} \left( \tau g_n^{p+2}+ \sum_{j=1}^{2}\mathcal{L}_{q_j}(\tau_jh_{jn},\alpha_jg_n)\right)\ dx+J_2 \\
&\ \ \leq C \|g_n\|_{L^{p+2}}^{p+2} + C \int_{-\infty}^{\infty}\sum_{j=1}^{2}\mathcal{L}_{q_j}(h_{jn},|g_n|)\ dx+C\left(1+\|\Delta_n\|_\h \right).
\end{aligned}
\end{equation}
The Gagliardo-Nirenberg inequality together with the estimate of $\|g_n\|_{L^2}^2$ as in \eqref{bound1} assures that
\begin{equation}\label{est1}
\|g_n\|_{L^{p+2}}^{p+2} \leq C\left(\|\Delta_n\|_{\h}^{p/2}+\|\Delta_n\|_{\h}^{(3p+4)/4} \right).
\end{equation}
Similarly, one can estimate
\begin{equation}\label{est2}
\begin{aligned}
\mathcal{L}_{q_j}(h_{jn},|g_n|)& \leq C \|\Delta_n\|_{\h}^{q_j/2}+C\|\partial_xh_{jn}\|_{L^2}^{1/2}\|g_n\|_{L^2}\\
& \leq C\left(1+\|\Delta_n\|_{\h}+\|\Delta_n\|_{\h}^{q_j/2}\right),\ j=1,2.
\end{aligned}
\end{equation}
Applying the estimates \eqref{est1} and \eqref{est2}, it follows from \eqref{mainest} that
\begin{equation*}
\|\Delta_n\|_{\h}^{2}\leq C\left(1+\|\Delta_n\|_\h+\|\Delta_n\|_{\h}^{p/2}+ \|\Delta_n\|_{\h}^{(3p+4)/4}+\sum_{j=1}^{2}\|\Delta_n\|_{\h}^{q_j/2} \right),
\end{equation*}
which in turn implies that $\|\Delta_n\|_\h$ is bounded.
\end{proof}

\noindent In the following lemma we
relate the solutions of \eqref{deftheta} to that of \eqref{Vardef}.

\begin{lem}\label{lem4existence}
Suppose that $(r,l,m)\in \mathbb{R}_{+}^2\times \mathbb{R},$ and let $b=b_{r,l,,m}(A)$ be as
defined in \eqref{bdef}. Then the following holds
\begin{equation}\label{decom2}
\Lambda(r,l,,m)=\inf\left\{\Theta(r,l,,A)+b^{2}(r+l):A\geq 0\right\}.
\end{equation}
Furthermore,
if $\{(h_{1n},h_{2n},g_n)\}\subset \h$ is a minimizing sequence for the problem $\Lambda(r,l,m),$ then there
exist a subsequence $\{(h_{1n_k}, h_{2n_k}, g_{n_k})\}$ and a number $A\geq 0$ such that the sequence
\begin{equation*}
\left\{\left(e^{ib_{r, l,m}(A)x}h_{1n_k},e^{ib_{r, l,m}(A)x}h_{2n_k}, g_{n_k}\right)\right\}
\end{equation*}
of functions in $\h$ forms a minimizing sequence for $\Theta(r,l,A).$ Moreover, we have that
\begin{equation}\label{relation1}
\Lambda(r,l,m)=\Theta(r,l,A)+b(r+l).
\end{equation}
Furthermore, one has $A>0$ provided that $\gamma_1=\gamma_2=0.$
\end{lem}
\begin{proof}
To prove \eqref{decom2}, suppose first that $A\geq 0$ and
let $(h_1,h_2,g)\in S_{r\times l}\times K_A$ be given.
Let $b=b_{r,l,m}(A)$ be as defined in \eqref{bdef} and%
\begin{equation*}
c_j=\textrm{Im}\int_{-\infty }^{\infty }h_j \overline{\partial_xh_{j}}\ dx,\ j=1,2.
\end{equation*}
Put $f_j(x)=e^{ik_jx}h_j(x)$ with $k_1=(c_1/r)-b$ and $k_2=(c_2/l)-b.$ Then,
for $\Delta=(f_1,f_2,g)$ and $U=(h_1,h_2,g),$ an elementary calculation gives
\begin{equation*}
H(\Delta) =H(U)- 2\sum_{j=1}^{2}k_j \|h_j\|_{L^2}^2= A+2(c_1+c_2)-2(k_1r+k_2l)=m.
\end{equation*}
Since $Q(f_1)=Q(h_1)=r$ and $Q(f_2)=Q(h_2)=l,$ we conclude that%
\begin{equation}\label{decomE}
\begin{aligned}
\Lambda(r,l,m) & \leq E(\Delta)=E(U)+\sum_{j=1}^{2} k_{j}^{2}\|h_j\|_{L^2}^2
-2\sum_{j=1}^{2}k_j \ {\rm Im} \int_{-\infty }^{\infty }h_j~\overline{\partial_xh_j}\ dx\\
& = E(U)+b^2(r+l)- \frac{c_1}{r}- \frac{c_2}{l} \leq E(U)+b^{2}(r+l).
\end{aligned}
\end{equation}
One can now take infimum over the set $S_{r\times l}\times K_A$ to obtain
\begin{equation}\label{InfInqA}
\Lambda(r,l,t)\leq \inf \left\{\Theta(r,l,A)+b^{2}(r+l):A\geq 0 \right\}.
\end{equation}
To obtain the reverse inequality, let $(r,l,m)\in \mathbb{R}_{+}^2\times \mathbb{R}$ be given and
$U=(h_1,h_2,g)\in \h$ be such that $(h_1,h_2)\in S_{r\times l}$ and $H(U)=m.$ We will
show that there exists $A\geq 0$ such that%
\begin{equation*}
E(U)\geq \Theta(r,l,A)+b^{2}(r+l).
\end{equation*}
Choose $A=\|g\|_{L^2}^{2}.$ Then, by the definition of $H,$ we have that
\begin{equation*}
A=m-2\sum_{j=1}^{2}\text{Im}\int_{-\infty }^{\infty }h_j\overline{\partial_xh_j}\ dx.
\end{equation*}
For $j=1,2,$ define $f_j(x)=e^{ib_{r,l,m}(A)x}h_j(x),$ where $b=b_{r,l,m}(A)$ is as defined in \eqref{bdef}.
Then, a straightforward calculation yields
\begin{equation*}
\begin{aligned}
E(\Delta) & =E(U)+\sum_{j=1}^{2}b^{2}\|h_j\|_{L^2}^2-2\sum_{j=1}^2 b \ \textrm{Im}\int_{-\infty }^{\infty }h_j\overline{\partial_xh_j}\ dx \\
&= E(U)+ b^2(r+l)- b(m-A)=E(U)- b^2(r+l),
\end{aligned}
\end{equation*}
from which it is obvious that $E(U)=E(\Delta)+b^2(r+l).$
Since $Q(f_1)=Q(h_1)=r$ and $Q(f_2)=Q(h_2)=l,$ and $%
g \in K_A,$ we have that $A \geq 0$ and $E(\Delta)\geq \Theta(r,l,A).$
In consequence, one has that
\begin{equation*}
\begin{aligned}
E(U)\geq \Theta(r,l,A) + b^2(r+l)\geq \inf_{A\geq 0}\left\{ \Theta(r,l,A)+ b^2(r+l)\right\}.
\end{aligned}
\end{equation*}
Upon taking infimum over all functions $U\in \h$ such that $(h_1,h_2)\in S_{r\times l}$ and $H(U)=m,$ we obtain the reverse inequality
\begin{equation}\label{InfInqB}
\Lambda(r,l,m) \geq \inf_{A\geq 0}\left\{ \Theta(r,l,A)+ b^2(r+l)\right\}.
\end{equation}
Putting the inequalities \eqref{InfInqA} and \eqref{InfInqB} together, we see that identity \eqref{decom2} holds.

\smallskip

\noindent Next, denote $\Delta_n=(h_{1n},h_{2n},g_n).$ The sequence $\{A_n\}$ of real numbers given by
\begin{equation*}
A_n=\|g_n\|_{L^2}^2=m-2\sum_{j=1}^{2}\text{Im}\int_{-\infty }^{\infty }h_{jn}\overline{\partial_xh_{jn}}\ dx
\end{equation*}
is bounded. Therefore, by extracting an appropriate subsequence, one may assume that $A_n$ converges to $A\geq 0.$
So by restricting consideration to the corresponding subsequence, let $b=b_{r, l,m}(A)$ and define $f_{jn}(x)=e^{ibx}h_{jn}(x).$
Denote $U_n=(f_{1n},f_{2n},g_n).$ Then one can invoke \eqref{decom2} to obtain
\begin{equation*}
\begin{aligned}
\lim_{n\to \infty}E(U_n)& =\lim_{n\to \infty}\left[E(\Delta_n)+b^2(r+l)-b(m-A_n) \right] \\
& = \Lambda(r,l,m)-b^2(r+l)\leq \Theta(r,l,m).
\end{aligned}
\end{equation*}
To obtain the reverse inequality, suppose first that $A>0.$ Then the sequences of
numbers $\alpha_{1n}=\sqrt{r}/\|f_{1n}\|_{L^2}^2, \alpha_{2n}=\sqrt{l}/\|f_{2n}\|_{L^2}^2,$ and
$\beta_n=\sqrt{A}/\|g_{n}\|_{L^2}^2$ are well-defined for
sufficiently large $n.$ Since $\|\alpha_{1n}f_{1n}\|_{L^2}^2=r,\ \|\alpha_{2n}f_{2n}\|_{L^2}^2=l,$ and
$\|\beta_ng_n\|_{L^2}^2=A,$ it follows immediately that
\begin{equation*}
\lim_{n\to \infty}E(f_{1n},f_{2n},g_n)=\lim_{n\to \infty}E(\alpha_{1n}f_{1n},\alpha_{2n}f_{2n},\beta_ng_n)\geq \Theta(r,l,m).
\end{equation*}
If $A=0,$ then one has that
\begin{equation*}
\lim_{n\to \infty}E(U_n) =\lim_{n\to \infty}\sum_{j=1}^{2}E_j(f_{jn})
= \Theta(r,l,0).
\end{equation*}
It now follows that the relation \eqref{relation1} holds and that $E(U_n)\to \Theta(r,l,m),$ this
means that $\{U_n\}$ is a minimizing sequence for $\Theta(r,l,m).$

\smallskip

Finally, consider the case when $\gamma_1=\gamma_2=0.$ Suppose, for the sake of contradiction, that $A=0.$ Then
\begin{equation*}
\Theta(\lambda_1,\lambda_2,0)=\inf\left\{\int_{-\infty}^{\infty}\left(|\partial_xf|^2+|\partial_xg|^2 \right)\ dx:\|f\|_{L^2}^2=\lambda_1,\ \|g\|_{L^2}^2=\lambda_2  \right\}.
\end{equation*}
It is clear that $\Theta(r,l,0)\geq 0$ and an application of \eqref{relation1} gives $\Lambda(r,l,m)\geq 0.$
On the other hand, let $\Delta_\theta$ be as considered in Lemma~\ref{Neginf}. Then one obtains $E(\Delta_\theta)<0,$ which in turn
implies that $\Lambda(r,l,m)<0,$ a contradiction.
\end{proof}

\subsection{Stability result for (2+1)-component NLS-gKdV}

We now prove Theorems~\ref{P2thm} and \ref{stabilitytheorem}.

\smallskip

\noindent \textbf{Proof of Theorem~\ref{P2thm} (existence result).} To prove part (i), using the same notation as in Lemma~\ref{lem4existence}, we may
assume by passing to an appropriate subsequence that $\{(e^{ibx}h_{1n},e^{ibx}h_{2n},g_n)\}$ is a
minimizing sequence for $\Theta(r,l,A),$ for $A\geq 0,$ $b=b_{r,l,m}(A),$ and the relation \eqref{relation1} holds.
If $A>0,$ then Theorem~\ref{SOexistence} allows us to conclude, again possibly for a subsequence only,
that there exists a family $(y_n)\subset \mathbb{R}$ such that
\begin{equation*}
\left\{ \left(e^{ib(\cdot+y_n)}h_{1n}(\cdot+y_n),e^{ib(\cdot+y_n)}h_{2n}(\cdot+y_n),g_n(\cdot+y_n)\right)\right\}
\end{equation*}
converges in $\h$ to some $U=(\phi_1,\phi_2,w)$. The same conclusion holds in the case when $A=0$ as well
(This can be easily checked using the identity obtained in the last paragraph of the proof of Lemma~\ref{lem4existence}.)
Furthermore, $U$ is a minimizing function for
$\Theta(r,l,A).$
For $j=1,2,$ by passing to an appropriate subsequence yet again, one may assume
that $e^{iby_n}\to e^{i\theta}$ for some number $\theta \in [0,2\pi).$
It then follows immediately that
\begin{equation*}
\left(h_{1n}(\cdot+y_n),h_{2n}(\cdot+y_n),g_n(\cdot+y_n) \right)\to (\Phi_1,\Phi_2,w),
\end{equation*}
in $\h,$ where $\Phi_j$ are given by $\Phi_j(x)=e^{-i(bx+\theta)}\phi_j(x).$ Let us denote $V=(\Phi_1,\Phi_2,w).$ Then, a calculation
similar to that made in \eqref{decomE} yields
\begin{equation*}
\begin{aligned}
\Theta(r,l,m)& =E(U)=E(V)+b^2\sum_{j=1}^2\|\Phi_j\|_{L^2}^2-2b\sum_{j=1}^2 \textrm{Im}\int_{-\infty}^{\infty}\Phi_j\overline{\partial_x\Phi_j}\ dx\\
& = E(V)+b^2(r+l)-b\left(H(V)-\|w\|_{L^2}^2\right)=E(V)-b^2(r+l).
\end{aligned}
\end{equation*}
Then, from the relation \eqref{relation1}, it follows that $V$ is a minimizing function
for the problem $\Lambda(r,l,m),$ which completes the proof.

\smallskip

To prove part (ii), let $(\Phi_1,\Phi_2,w)$ be a solution of \eqref{Vardef}.
By the first part of Lemma~\ref{lem4existence}, it follows that
$(e^{ibx}\Phi_1,e^{ibx}\Phi_2,w)$ is a minimizing
sequence (and hence a minimizer) for $\Theta(r,l,\|w\|_{L^2}^2),$ where $b$ is as defined in \eqref{bdef} with $A=\|w\|_{L^2}^2.$
Then, invoking Theorem~\ref{SOexistence}, there
exists $(\theta_j,\phi_j)\in \mathbb{R}\times H_{+}^1(\mathbb{R})$ such that
\begin{equation*}
\left(e^{ibx}\Phi_1, e^{ibx}\Phi_2 \right)=\left( e^{i \theta_1}\phi_1, e^{i\theta_2}\phi_2\right).
\end{equation*}
Furthermore, if $\gamma_1=\gamma_2=0,$ then the last assertion of Lemma~\ref{lem4existence} implies that $A>0.$
Since $(\phi_1,\phi_2,w)$ belongs to $\mathcal{O}_{r,l,A},$ Theorem~\ref{SOexistence} guarantees
that $w(x)>0$ for $x\in \mathbb{R}.$ \hfill $\Box$

\medskip

Finally, we prove the stability result.

\smallskip

\textbf{Proof of Theorem~\ref{stabilitytheorem} (stability result).} Part (i) is an easy consequence of the
existence result (Theorem~\ref{P2thm}).
To prove part (ii),
suppose that $\mathcal{P}_{r,l,m}$ is not stable.
Then there exists a sequence of solutions $\{(u_{1n},u_{2n},v_n)\}$ of \eqref{3LSI} and a sequence of times $\{t_n\}$ such that
$(u_{1n}(\cdot,0),u_{2n}(\cdot,0),v_n(\cdot,0))$ converges to $\mathcal{P}_{r,l,m},$ but $(u_{1n}(\cdot,t_n),u_{2n}(\cdot,t_n),v_n(\cdot,t_n))$
does not converge to $\mathcal{P}_{r,l,m}$ in $\h.$ Since $E,Q,$ and $H$ are
constants of the motion of \eqref{3LSI} and are continuous on $X,$ it follows that
\begin{equation*}
\begin{aligned}
& \lim_{n\to \infty}Q(u_{1n}(\cdot,t_n))= r, \ \ \lim_{n\to \infty}Q(u_{2n}(\cdot,t_n))= l, \\
& \lim_{n\to \infty} H(u_{1n}(\cdot,t_n),u_{2n}(\cdot,t_n),v_n(\cdot,t_n)) = m,\ \textrm{and} \\
& \lim_{n\to \infty} E(u_{1n}(\cdot,t_n),u_{2n}(\cdot,t_n),v_n(\cdot,t_n)) = \Lambda(r,l,m).
\end{aligned}
\end{equation*}
Hence, from part (i), it follows that $(u_{1n}(\cdot,t_n),u_{2n}(\cdot,t_n),v_n(\cdot,t_n))$
converges to $\mathcal{P}_{r,l,m}$ in $\h,$ which is a contradiction.

\smallskip

To prove part (iii),
suppose $(\Phi_1,\Phi_2,w_1)\in \mathcal{P}_{r_1,l_1,m_1}$
and $(\Phi_3,\Phi_4,w_2)\in \mathcal{P}_{r_2,l_2,m_2},$ where $(r_1,l_1,m_1)\neq (r_2,l_2,m_2).$
 We wish to prove that $(\Phi_1,\Phi_2,w_1)\neq (\Phi_3,\Phi_4,w_2).$
 If $r_1\neq r_2,$ then the desired conclusion is clear. So assume that $r_1=r_2$ and $m_1\neq m_2.$
Let us denote
\begin{equation*}
\eta_1=\frac{\|w_1\|_{L^2}^2-m } {2(r_1+l_1)}\ \textrm{and}\ \eta_3=\frac{\|w_2\|_{L^2}^2-m } {2(r_2+l_2)}.
\end{equation*}
Then, part (ii) of Theorem~\ref{P2thm}, there exists a pair of real
numbers $\theta_1, \theta_3$ and a pair of $\mathbb{R}$-valued functions $\phi_1, \phi_3$ such that
\begin{equation}\label{true3}
\begin{aligned}
\Phi_1(x)=e^{i(\eta_1x+\theta_1)}\phi_1(x)\ \textrm{and}\ \Phi_3(x)=e^{i(\eta_3x+\theta_3)}\phi_3(x).
\end{aligned}
\end{equation}
One may assume that $\Phi_1=\Phi_3,$ since otherwise the desired conclusion follows.
Then \eqref{true3} implies that
 \begin{equation*}
 e^{i((\eta_1-\eta_3)x+(\theta_1-\theta_3)) }= \phi_3(x)/ \phi_1(x)
 \end{equation*}
 is a $\mathbb{R}$-valued function on $\mathbb{R},$ and hence we must have $\eta_1=\eta_3.$
 Since $r_1=r_2,$ this in turn gives
 \begin{equation} \label{trueaux}
 l_2\left(\|w_1\|_{L^2}^2-m_1\right)= l_1\left(\|w_2\|_{L^2}^2-m_2\right).
 \end{equation}
 If $l_1\neq l_2,$ then $\|\Phi_2\|_{L^2}^2\neq \|\Phi_4\|_{L^2}^2,$ and hence $\Phi_2\neq \Phi_4,$ the conclusion follows.
 So we may assume that $l_1=l_2.$ Then \eqref{trueaux} implies
 $\|w_1\|_{L^2}^2-m_1= \|w_2\|_{L^2}^2-m_2$ and since $t_1\neq m_2,$ this in turn implies
  $\|w_1\|_{L^2}^2\neq \|w_2\|_{L^2}^2,$ and hence $w_1\neq w_2.$ This completes the proof of Theorem~\ref{stabilitytheorem}.

\bigskip
\noindent
{\bf Acknowledgment.} The authors are thankful to Professor John Albert and Professor Felipe Linares for their teaching. A. J. Corcho was partially supported by CAPES and
CNPq/Edital Universal - 481715/2012-6, Brazil and  M. Panthee acknowledges supports from Brazilian agencies FAPESP 2012/20966-4 and CNPq 479558/2013-2 (Edital Universal)  \& 305483/2014-5.


\medskip

\begin{thebibliography}{99}

\bibitem{[AA]} J. Albert and J. Angulo, {\em Existence and stability
of ground-state solutions of a Schr\"{o}dinger-KdV system}, Proc. of the
Royal Soc. of Edinburgh, \textbf{133A} (2003) 987--1029.

\bibitem{[AB11]} J.~Albert and S.~Bhattarai, {\em Existence and stability of a two-parameter
family of solitary waves for an NLS-KdV system}, Adv. Differential Eqs.,  {\bf 18} (2013) 1129 -- 1164.

\bibitem{[AP1]} J.~Angulo, {\em Stability of solitary wave solutions
for equations of short and long dispersive waves}, Elec. J. of Diff. Eqns.
{\bf 72} (2006) 1--18

\bibitem{[Bar]} T.~Bartsch and L.~Jeanjean, {\em Normalized solutions for nonlinear Schr\"{o}dinger systems}, preprint, arXiv:1507.04649.

\bibitem{BOP} D. Bekiranov, T. Ogawa and G. Ponce, {\em Interaction equations for short and long dispersive
waves}, J. Funct. Anal. {\bf 158} (1998) 357--388.

\bibitem{[Benj]} T. B. Benjamin, {\em The stability of solitary waves}, Proc. Roy. Soc. London Ser. A, \textbf{328} (1972), 153--183.

\bibitem{[San3LS]} S.~Bhattarai, {\em Existence and positivity properties of solitary waves for a
          multicomponent long wave-short wave interaction system}, preprint, arXiv:1508.07598

\bibitem{[SanDCDS]} S.~Bhattarai, {\em Stability of normalized solitary waves for three coupled nonlinear Schr\"{o}dinger equations},
       to appear in Discre. Contin. Dyn. Syst. arXiv:1509.00425

\bibitem{[Ca]} T.~Cazenave, {\em Semilinear Schr\"odinger equations}, Courant Lecture Notes in Mathematics, vol.\ {\textbf 10},
American Mathematical Society, Providence, 2003.

\bibitem{CLi} T.~Cazenave and P.-L. Lions, {\em Orbital stability of standing waves for some nonlinear Schr\"{o}dinger
               equations}, Comm. Math. Phys. {\bf 85} (1982) 549--561.

\bibitem{[C]} L. Chen, {\em Orbital stability of solitary waves of the nonlinear Schr\"{o}dinger-KDV equation}, J.
Partial Diff. Eqs. \textbf{12} (1999), 11--25.

\bibitem{CKSTT-1} J. Colliander, M. Keel, G. Staffilani, H. Takaoka and T. Tao; {\em Global well-posedness
for the KdV in Sobolev spaces of negative indices}, Elect. J. Differ. Eqns., {\bf 26} (2001)
1--7.

\bibitem{CKSTT-2} J. Colliander, M. Keel, G. Staffilani, H. Takaoka and T. Tao; {\em Sharp global wellposedness
for KdV and modified Kdv on $\R$ and $\T$}, J. Amer. Math. Soc. {\bf 16} (2003)
705--749.

\bibitem{[Colo]} E.~Colorado, {\em On the existence of bound and ground states for some
  coupled nonlinear Schr\"{o}dinger--Korteweg-de Vries equations}, preprint, arXiv:1411.7283.

\bibitem{CL} A. J. Corcho and F. Linares; {\em Well-posedness for the Schr\"odinger-Korteweg-de Vries system}, Trans. AMS {\bf 359} (2007) 4089--4106.

\bibitem{Te} T. Esteves; {\em O problema de Cauchy para a equa\c c\~ao KdV super-sim\'etrica com dado inicial pequeno}, Master Dissertation. Federal University of Piaui, Brazil (2014).

\bibitem{[Gar1]} D. Garrisi, {\em On the orbital stability of standing-waves
            pair solutions of a coupled non-linear Klein-Gordon equation}, Adv. Nonlinear Studies, \textbf{12} (2012) 639--658.

\bibitem{GTV} J. Ginibre, Y. Tsutsumi and G. Velo; {\em On the Cauchy problem for the Zakharov system}, J. Funct. Anal. {\bf 151} (1997) 384--436.

\bibitem{Guo} B. Guo and C. Miao; {\em Well-posedness of the Cauchy problem for the coupled system
of the Schr\"{o}dinger-KdV equations}, Acta Math. Sinica, Engl. Series {\bf 15} (1999) 215--224.

\bibitem{[Ikoma]} N.~Ikoma, {\em Compactness of minimizing sequences in nonlinear Schr\"{o}dinger systems under
                multiconstraint conditions}, Adv. Nonlinear Studies {\bf 14} (2014) 115--136.

\bibitem{KPV-a}  C. E. Kenig,  G. Ponce and  L. Vega; {\em Well-posedness of the initial value problem
for the Korteweg-de Vries quation}, J. Amer.Math. Soc., {\bf 4}   (1991), 323--347.

\bibitem{KPV} C. E. Kenig, G. Ponce and L. Vega; \emph{A bilinear estimate with applications to the KdV equation},
J. Amer. Math. Soc. \textbf{9  2} (1996)  573--603.

\bibitem{[L1]} P.~L.~Lions, {\em The concentration-compactness principle in
the calculus of variations. The locally compact case, Part 1}, Ann. Inst. H. Poincare
Anal. Non-lin\'{e}aire {\bf 1} (1984) 104--145.

\bibitem{[LL]} E. H. Lieb, M. Loss, {\em Analysis}, 2nd ed., Graduate studies in mathematics, {\bf  14} American
Mathematical Society, Providence, 2001.

\bibitem{[Mo]} L.~Molinet and F.~Ribaud, {\em  Well-Posedness Results for the Generalized Benjamin-Ono Equation with Arbitrary Large Initial Data}, IMRN International Math. Res. Notices No 70 (2004) 3757--3795.

\bibitem{[Mu]} C.~Muscalu and  W.~Schlag, {\em Classical and Multilinear Harmonic Analysis}, Vol. I Cambridge (2013).

\bibitem{[Noris]} B.~Noris, H.~Tavares, G.~Verzini, {\em Stable solitary  waves  with  prescribed $L^2$-mass  for  the
cubic Schr\"{o}dinger system with trapping potentials},  Discr. Contin. Dyn. Syst. A {\bf 35} (2015) 6085--6112.

\bibitem{[Pec]} H. Pecher, {\em The Cauchy problem for a  Schr\"{o}dinger-Korteweg-de Vries system with
rough data}, Diff. Int. Eqns. {\bf 18} (2005) 1147--1174.

\bibitem{[Tsu]} M.~Tsutsumi, {\em Well-posedness of the Cauchy problem for a coupled Schr\"{o}dinger-KdV equation},
Math. Sciences Appl. {\bf  2} (1993) 513--528.

\bibitem{WU} Y. Wu, {\em The Cauchy problem of the Schr\"odinger-Korteweg-de Vries system}, Diff. Int. Equations {\bf 23}  (2010) 569--600.

\end{thebibliography}
\end{document}